\documentclass[11pt,a4paper]{amsart}
\usepackage{geometry}
\usepackage[utf8]{inputenc}
\usepackage[T1]{fontenc}
\usepackage{lmodern}
\usepackage[english]{babel}
\usepackage{amsmath,mathtools,amsthm,amssymb,amsfonts}
\usepackage{enumerate}
\usepackage{color}
\usepackage{mathrsfs}
\usepackage{graphicx}
\usepackage{tikz} 
\usepackage{pgfplots}
\usepackage{esvect}
\usepackage{pdfpages}
\usetikzlibrary{shapes,arrows}
\usetikzlibrary{calc}
\usetikzlibrary{shapes.geometric}
\usetikzlibrary{shapes.arrows}
\usetikzlibrary{arrows.meta}
\usetikzlibrary{decorations.markings}
\usetikzlibrary{fit}
\usetikzlibrary{patterns}
\usetikzlibrary{hobby}
\usepgfplotslibrary{patchplots}
\pgfplotsset{compat=1.11}
\usepackage{hyperref}
\usepackage{nicefrac}
\usepackage{cancel}
\usepackage{parskip} 

\theoremstyle{plain}
\newtheorem{theorem}{Theorem}[section]

\newtheorem{lemma}[theorem]{Lemma}
\theoremstyle{definition}
\newtheorem{definition}[theorem]{Definition}
\newtheorem{remark}[theorem]{Remark}
\newtheorem*{theorem*}{Theorem}
\newtheorem*{remark*}{Remark}
\newtheorem*{thm*}{Theorem}
\newtheorem*{conjecture*}{Conjecture}

\newcommand{\definedas}{\mathrel{\raise.095ex\hbox{\rm :}\mkern-5.2mu=}}
\newcommand{\asdefined}{\mathrel{=\mkern-5.2mu\raise.095ex\hbox{\rm :}}}

\newcommand{\R}{\mathbb{R}}
\newcommand{\IS}{\mathbb{S}}

\newcommand{\diver}{\operatorname{div}}
\newcommand{\tr}{\operatorname{tr}}

\newcommand{\Hess}{\operatorname{Hess}}
\newcommand{\diam}{\operatorname{diam}}
\newcommand{\IN}{\mathbb{N}}
\newcommand{\qtext}[1]{\quad\text{#1}\quad}
\newcommand{\eps}{\varepsilon}

\newenvironment{enum}{\begin{enumerate}[i)]}{\end{enumerate}}

\newcounter{flabelcounter}
\newenvironment{steps}
{\begin{list}{}{\setlength\labelwidth{0pt}%
      \setlength\leftmargin{0pt}
      \setlength\itemindent{0pt}%
      }}
  {\end{list}}

\newcommand{\ds}{\,\mathrm{d}s}
\newcommand{\dt}{\,\mathrm{d}t}
\newcommand{\du}{\,\mathrm{d}u}

\allowdisplaybreaks

\title[Local Foliations]{Local Space Time Constant Mean Curvature and Constant  expansion foliations}

\author[Metzger]{Jan Metzger}
\address{University of Potsdam, 14476 Potsdam, Germany}
\email{jan.metzger@uni-potsdam.de}

\author[Pe\~nuela Diaz]{Alejandro Pe\~nuela Diaz}
\address{Max-Planck Institute for Gravitational Physics, University of Potsdam, 14476 Potsdam, Germany}
\email{alejandro.penuela@aei.mpg.de}

\begin{document}
\begin{abstract}
    Inspired by the small sphere-limit for quasi-local energy we study local foliations of surfaces with prescribed mean curvature. Following the strategy used by Ye in \cite{Ye} to study local constant mean curvature foliations  we use a Lyapunov Schmidt reduction  in an $n+1$ dimensional manifold equipped with a symmetric $2$-tensor to construct the foliations around a point, prove their uniqueness and show their nonexistence conditions. To be specific, we study two foliation conditions. First we consider constant space-time mean curvature surfaces. They are used to characterizing the center of mass in general relativity \cite{STCMC}. Second, we study local foliations of constant expansion surfaces \cite{Ce}. 
\end{abstract}
	
\maketitle
	
\section{Introduction and Results}

The search for a good definition of center of mass in general relativity has been a long standing problem in physics, with many different attempts and possible definitions given by both physicists and mathematicians. In 1996 by Huisken and Yau \cite{HY} proposed to use a foliation of constant mean curvature $2$-spheres (CMC surfaces) in an asymptotically flat manifolds. Under suitable conditions these foliations are unique and fully characterize the center of mass at spatial infinity of an isolated physical system in the time symmetric case. This result was later refined and improved by the first named author \cite{Ce}, by Huang \cite{Huang2, Huang1}, Eichmair and the first named author \cite{Eich} and Nerz \cite{Nerz2} among others. The non time symmetric case has been a more elusive problem with different attempts. Some of these were trying to find a generalization to the CMC foliation of Huisken and Yau like the constant expansion foliation or the STCMC foliation proposed by Cederbaum and Sakovich \cite{STCMC}. For a more extensive exposition of the problem see \cite{wangyau3,wangyau2} and  \cite{Nerz3}. 


\subsection*{Surfaces of constant expansion (CE)}
A a $3$-dimensional initial data set $(M,g,k)$ for General Relativity is a $3$-dimensional Riemannian manifold $(M,g)$ and a symmetric $2$-tensor $k$ that plays the role of the second fundamental form, when $(M,g)$ is embedded as a space-like hyper-surface in a $4$-dimensional space-time. If $(M,g,k)$ is asymptotically flat in a suitable sense and when the ADM-mass of $(M,g,k)$ is positive then the first named author found in \cite{Ce} that there are two foliations by $2$-spheres $\{\Sigma_r^\pm \}_{r>r_0}$ of constant expansion. These surfaces satisfy the prescribed mean curvature equation 
\begin{equation*}
  \theta^\pm(\Sigma_r^\pm):= H(\Sigma_r^\pm) \pm P(\Sigma_r^\pm) = \frac{2}{r}.
\end{equation*}
Here $H$ denotes the mean curvature of the surface, $g_{\Sigma_r}$ the induced metric and $P=\tr_{g_{\Sigma_r}} k$ is the trace of the tensor $k$ restricted to the tangent space of the surface.  

The main motivation of this generalization of the result of Huisken and Yau was to include the dynamical term $k$ into the construction of the foliation. However, the order on which this affects the foliation implies that these surfaces do not have a well defined center of mass at infinity in the presence of ADM-momentum. In our results, described below, this feature also enters again. In fact, the existence of the foliation here depends on the local constant expansion 1-form \eqref{eq:ce-one-form0} which only includes information on $k$ and does not explicitly depend on the curvature of $g$. 

\subsection*{Surfaces of Constant Space-time Mean Curvature (STCMC)}
The work of Cederbaum and Sakovich \cite{STCMC} proposed a different generalization of the CMC foliation and found what appears to be the surfaces that characterize the center of mass of an isolated system. They show that for an asymptotically flat initial data set with non-vanishing ADM energy, there exist an unique foliation of $2$-spheres $\{\Sigma_r\}_{r>r_0}$ of constant space-time mean curvature (STCMC). These surfaces satisfy
\begin{equation*}
  \theta^+(\Sigma_r)\cdot \theta^-(\Sigma_r) = H^2(\Sigma_r) - P^2(\Sigma_r) = \frac{4}{r^2},   
\end{equation*}
that is, the product of both expansions of the surfaces is constant.
This foliation is unique and with some additional assumptions it has a well defined center of gravity for the surfaces at infinity. This center has all the required properties of a center of mass, that is the proper transformation and conservation laws, and generalizes the center of mass expression proposed by Beig and O'Murchadha \cite{beigmurch}. In particular it remedies the deficiencies described in \cite{Nerz3}. If an additional requirement on the divergence of the extrinsic torsion of the sphere is satisfied then a STCMC sphere is in particular a ``rigid'' sphere in the notion of Gittel, Jezierski, and Kijowski \cite{War1,War2}.
 
The motivation for this paper comes from the question, whether the surfaces studied in the asymptotically flat setting can be related to a physically meaningful quantity in the local case, where the diameters of the surfaces are very small, this would be similar to how quasi-local energies have an associated physical quantity when studied in the small sphere limit. In the literature there exist many definitions of quasi-local energies, as it can be seen in the detailed review \cite{liv}. But any such notion should satisfy certain physically motivated properties. In particular, they should have the right asymptotics in the large-sphere and the small-sphere regime. For the large sphere limits one should recover the ADM-mass or the Bondi-mass as the limit when evaluating the mass on spheres with radius increasing to infinity. Here, we consider the regime of the small sphere limit, when taking spheres with radius going to zero cut out of a light-cone centered at a point $p$ of space-time. The leading order term should recover the stress-energy tensor in space-times with matter fields. There are many results in this direction, like for instance  for the Hawking energy \cite{Haw}, the Brown-York energy \cite{Brown} the Kijowski-Epp-Liu-Yau  energy \cite{Yu}, the Wang-Yau \cite{wangyau} and for their higher dimensional  versions \cite{Wang}.

\begin{figure}[h]
\centering
\includegraphics[width=8.3cm]{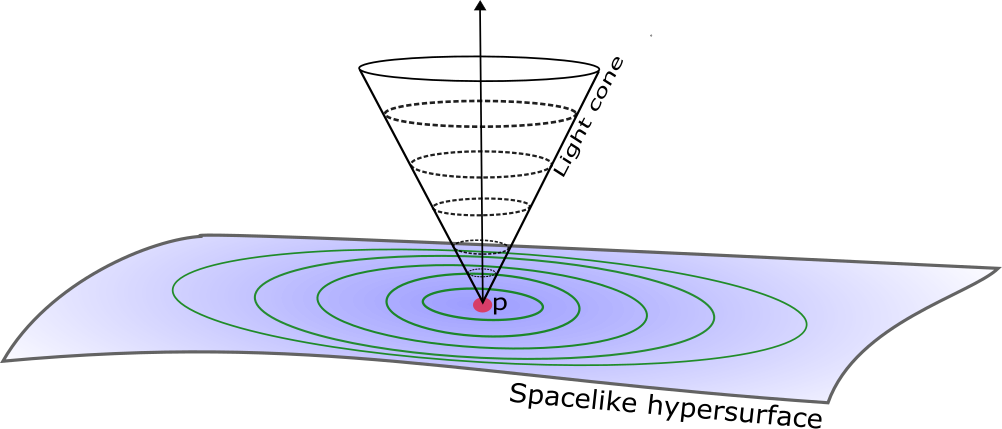}
\caption{ Comparison between approaching a point along cuts on a null cone and along a foliation on an initial data set (approach we are using  here).}
\end{figure}
This small sphere limit motivates the question if there is also an analogous condition for the center of mass (in case it is given by a foliation), and what condition should a point fulfill to possess or not to possess a concentration or a foliation of surfaces that characterize the center of mass. This question was actually solved for the time symmetric case by Ye \cite{Ye} in 1991. Ye showed that there is an unique CMC foliation around any non-degenerate critical point of the scalar curvature in a Riemannian manifold of arbitrary dimension. He also showed that if a point is not a critical point of the scalar curvature then such a foliation does not exist. Ye's ansatz is to construct the foliation as a family of perturbations of geodesic spheres. In flat space the linearization of the mean curvature operator has a three-dimensional kernel. The presence of this kernel leads to the integrability conditions that the scalar curvature at the center point is critical, and the non-degeneracy allows to solve for a small translation compensating the kernel using a Ljapunov-Schmidt reduction. This method has been used recently to show  some other interesting results like the construction of foliations of large area-constrained Willmore surfaces in  manifolds asymptotic to Schwarzschild \cite{eichko} or the finding of local area-constrained Willmore foliations in \cite{ikoma} and \cite{locwill}. We will use Ye's  strategy to study local foliations of STCMC  and CE surfaces. 

\begin{remark}
 Note that we will work only on an $n+1$-dimensional initial data set $(M,g,k)$ and approach a point $p\in M$ with deformed $n$-dimensional spheres contained in $M$. This is different from the method used to study the small sphere limit of quasi-local energy. There one approaches a point with suitable $n$-dimensional sections of its light cone in an $n+2$-dimensional ambient space-time.
\end{remark}

\subsection{Results for STCMC foliations}

In the following $(M,g,k)$ denotes an $n+1$-dimensional initial data set. The foliations considered will be regularly centered in the following sense:

\begin{definition}[\protect{cf.~\cite[Definition 1.1]{Ye}}]
    \label{regcent}
    A smooth foliation $\mathcal{F} $ of codimension~$1$ of $U \setminus \{ p \} $ for a neighborhood $U$ of $p$ is called a \emph{foliation centered at p}, provided that its leaves are all closed. If furthermore 
    \begin{equation}\label{regular}
      \sup_{S \in \mathcal{F} } \big( \sup_{S} |B_S |\cdot \diam S \big)<\infty,  
    \end{equation}
    then $\mathcal{F}$ is called a  \emph{foliation regularly centered at p}. Here $B_S$ denotes the second fundamental form of $S$, $|B_S|^2= {B_S}_{ij} {B_S}_{pq}\, {g_S}^{ip} {g_S}^{jq}$ and  $\diam S$ is the extrinsic diameter of the surface. 
\end{definition}
 We also consider \emph{concentrations of surfaces regularly centered at $p$}. 
 \begin{definition}
 We say that a family of closed compact embedded surfaces $\{S_r: r \in I \}$, where $I$ is an interval satisfying $0 \in \Bar{I} $, is a \emph{concentration of surfaces around $p$} if 
    \begin{equation*} 
        \limsup_{r \to 0} \diam S_r =0
        \quad \text{ and} \quad \bigcap_{r_0 \in (0,\infty)} \overline{\bigcup_{r \in I \cap (0,r_0) }S_r}= \{p\}.
    \end{equation*}
    Furthermore we say that the concentration is regularly centered if \eqref{regular} holds.
\end{definition}
Note that a foliation is a concentration of surfaces where the surfaces can be continuously parameterized by $r$ (that is $\forall r\in I$ there is a surface $S_r$) and where the surfaces do not intersect with each other.
\begin{definition}
    For any tangent vector $V \in TM$ we define the \emph{local STCMC $1$-form}
    \begin{equation*}
    \begin{split}
    A_\text{ST}(V)  =\frac{n}{2} \nabla_V \mathrm{Sc} + \frac{1}{(n+5)}&\Big[ \nabla_V (|k|^2 )   +((n+1)(n+5)+1)\nabla_V \Big(\frac{(\tr k)^2}{2}\Big)\\   &+4 \diver ( k^2 )(V)  - 2 (n+4) \diver \left(\tr k\cdot k\right)(V)\Big]
    \end{split}
    \end{equation*}
    where we denote $k^2_{jq}= k_{ij} g^{ip} k_{pq} $   and $|k|^2= k_{ij}k_{pq} g^{ip}g^{jq}$.
\end{definition}

\begin{theorem*}
 Let  $p\in M$ be such that $A_\text{ST}(p)=0$ and such that $\nabla A_\text{ST}(p)$ is invertible. Then there exists $C>0$ depending only on the dimension of $M$ with the following property: If 
 \begin{equation*}
     C |(\nabla A_\text{ST})^{-1}|\, (|k|^2 + |\mathrm{Ric}|  ) |k| |\nabla k|<1
     \quad\text{at}\quad p,
 \end{equation*}  
 then there exist a $\delta>0$ and a smooth foliation $\mathcal{F}= \{ S_r : r\in (0, \delta)  \}$ of STCMC spheres with $\sqrt{H_{S_r}^2-P_{S_r}^2}= \frac{n}{r}$. $\mathcal{F}$ is a foliation regularly centered at $p$ and its leaves $S_r$ are normal graphs over geodesic spheres of radius $r$. 
\end{theorem*}

\begin{theorem*}[Nonexistence and uniqueness of STCMC foliations]\ 
\begin{enum}
\item \label{i:stcmc-nonexistence-1} $A_\text{ST}(p) \neq 0$ for some point $p\in M$, then there exist no constant STCMC foliation regularly centered at $p$.
\item \label{i:stcmc-nonexistence-2} Let $p\in M$ be such that $A_\text{ST}(p)=0$, $\nabla A_\text{ST}(p)$ is invertible such that the foliation $\mathcal{F}$ of the previous theorem exists. If  $\mathcal{F}_2$ is a STCMC foliation regularly centered at $p$, then either $\mathcal{F}$ is a restriction of $\mathcal{F}_2$ or $\mathcal{F}_2$ is a restriction of $\mathcal{F}$.
\item Claims (\ref{i:stcmc-nonexistence-1}) and (\ref{i:stcmc-nonexistence-2}) also hold if instead of foliations we consider concentrations of regularly centered surfaces around $p$. 
\end{enum}
\end{theorem*}

\subsection*{Question}
With this results one can wonder if the local STCMC $1$-form characterizes the local behaviour of any possible definition of center of mass, that is if any reasonable definition of center of mass given by surfaces should concentrate around points with vanishing local STCMC $1$-form ($\mathrm{A_{ST}}_p=0$) when considering it as a local foliation or in the small sphere limit.

\subsection{Results for CE surfaces}
Even if the constant expansion foliations are not suitable to describe the center of mass of a system, this equation is interesting from a mathematical viewpoint. Indeed, there are two differnet ways to construct this foliations leading to different foliation conditions. We expect that this results also can help to understand the reason why these surfaces are not suitable to describe a center of mass, even if they look like the most direct generalization to the CMC surfaces.  
\begin{definition}

For any tangent vector $V \in TM$ we define the \emph{local constant expansion $1$-form}
\begin{equation}\label{eq:ce-one-form0}
\begin{split}
 \mathrm{\mathrm{A_{CE}}}(V)  =\frac{n+2}{n+3} \nabla_V \tr k -2 \diver k(V) 
  \end{split}
\end{equation}

and we also introduce  the \emph{second local constant expansion $1$-form}
\begin{equation}
\label{eq:ce-one-form}
\begin{split}
  \mathrm{\hat{A}_{CE}}^\pm(V)  =&-\frac{1}{2} \nabla_V \mathrm{Sc} \pm  \frac{1}{3(n+3)(n+5)}\Big( \frac{2(n^2 +6n +10)}{(n+3)} \langle \mathrm{Ric}(V, \cdot) ,\nabla \tr k \rangle - 4 \langle \mathrm{Ric} ,\nabla  k(V,\cdot) \rangle\\
   &-2 \langle \mathrm{Ric},\nabla_V  k \rangle  - \frac{n^3 +14n^2 +52n +60}{n(n+3)} \mathrm{Sc} \nabla_V \tr k \Big)  
  \end{split}
\end{equation}
\end{definition}
 \begin{theorem*}
 Let  $p\in M$ be such that at $p$, $\mathrm{\mathrm{A_{CE}}}=0$, $\nabla \mathrm{\mathrm{A_{CE}}} $ is invertible  and $k=0$. If in addition $\frac{1}{(n+3)} |(\nabla \mathrm{A_{CE}})^{-1}|\, |   \mathrm{\hat{A}_{CE}}^\pm| <1  $ at $p$, then exist $\delta^\pm>0$ and a smooth foliation $\mathcal{F}^\pm= \{ S_r^\pm : r\in (0, \delta^\pm)  \} $ of constant expansion  spheres with $H_{S_r^{\pm}}\pm P_{S_r^\pm}= \frac{n}{r}$. This foliation  $\mathcal{F}^\pm$ is regularly centered at $p$ and its leaves $S_r^\pm$ are normal graphs over geodesic spheres of radius~$r$. 
\end{theorem*}

  This result  is not a generalization of the CMC foliation of \cite{Ye}. However one can find a constant expasion foliation that generalizes the CMC one by adding some extra assumptions.
 
 \begin{definition}
 For $V,W \in TM$ define the $2$-tensor:  
\begin{equation}
    \begin{split}
    \hat{T}(V,W) :=&\frac{4}{(n+ 3)(n+5)} \big( \langle \nabla_W k,  2\nabla  k(V, \cdot)+  \nabla_V \, k \rangle - \frac{2n +5}{(n+3)^2} (\nabla_V \tr k \nabla_W \tr k \\
    &+ 2 \langle \nabla_W k(V, \cdot),  \nabla \tr k \rangle ) \big)
    \end{split}
\end{equation}
 \end{definition}
\begin{theorem*}
 Let  $p\in M$ be  such  that at $p$, $\mathrm{A_{CE}}=  \mathrm{\hat{A}_{CE}}^\pm =0$, $k= \nabla \mathrm{A_{CE}}=0$, $ \Hess\mathrm{A_{CE}}=0 $ and $\nabla   \mathrm{\hat{A}_{CE}}^\pm  + \hat{T}$  is invertible, then there exist a constant $C $ depending on the dimension of $M$ such that if at $p$, $ C |(\nabla   \mathrm{\hat{A}_{CE}}^\pm  + \hat{T})^{-1}| \left( |\nabla k|\,( |\mathrm{Ric}|+|\nabla k|  +|\nabla \nabla k|) +|\nabla \nabla \nabla k| \right)  <1  $ then there exist  $\delta^\pm>0$ and a smooth foliation $\mathcal{F}^\pm= \{ S_r^\pm : r\in (0, \delta^\pm)  \} $ of constant expansion  spheres with $H_{S_r^\pm}\pm P_{S_r^\pm}= \frac{n}{r}$. $\mathcal{F}^\pm$ is a foliation regularly centered at $p$ and its leaves $S_r^{\pm}$ can be express as normal graphs over geodesic spheres of radius $r$. 
\end{theorem*}

\begin{theorem*}[Nonexistence and uniqueness constant expansion]\
\begin{enum}
    \item
    If at a point $p$, $\mathrm{A_{CE}} \neq 0$ or $k \neq 0$, then there exist no constant expansion foliation regularly centered at $p$.
    \item Assume that at $p$, $\mathrm{A_{CE}}=0$, $k=0 $, $\nabla A_{CE}$ is invertible and that the foliation of the antepenultimate theorem $ \mathcal{F}$ exist. If  $ \mathcal{F}_2$ is a constant expansion foliation regularly centered at $p$, then either $\mathcal{F}$ is a restriction of $\mathcal{F}_2$ or $\mathcal{F}_2$ is a restriction of $\mathcal{F}$.
    \item Assume that at $p$, $\mathrm{A_{CE}}=  \mathrm{\hat{A}_{CE}}^\pm =0$, $k= \nabla \mathrm{A_{CE}}=0$, $ \Hess\mathrm{A_{CE}}=0 $, $\nabla   \mathrm{\hat{A}_{CE}}^\pm  + \hat{T}$  is invertible, and that the foliation of the previous theorem  $ \mathcal{F}$ exist. If  $ \mathcal{F}_2$ is a constant expansion foliation regularly centered at $p$, then either $\mathcal{F}$ is a restriction of $\mathcal{F}_2$ or $\mathcal{F}_2$ is a restriction of $\mathcal{F}$.
    \item The previous claims also hold if instead of foliations we consider     concentrations of regularly centered surfaces around $p$. 
\end{enum}
\end{theorem*}

\begin{remark}
    Note that $k$ plays a more prominent role in this foliations, so that we have to ask for stronger  conditions on $k$. In \cite{Nerz} Nerz reduced the decay conditions required for asymptotically flat manifolds to have a constant expansion foliation, when comparing this result to the STCMC foliation one sees that the constant expansion one requires a higher decay of $k$ in order to exist. This feature is reflected here, since our condition that $k=0$ could be seen as an analogon of this  higher decay.  
\end{remark}

\paragraph*{\emph{Acknowledgements.}} 
The first author acknowleges support from the SPP 2026 in form of the DFG project ME 3816/3-1. The second author is supported by the International Max Planck Research School for Mathematical and Physical Aspects of Gravitation, Cosmology and Quantum Field Theory.

\section{Preliminaries}

We work with data $(M,g,k)$ where $(M,g)$ is a smooth $n+1$-dimensional Riemannian manifold which is equipped with a symmetric $2$-tensor $k$. In General relativity, the data $(M,g,k)$ represents a space-like hypersurface (or an inital data set) with second fundamental form $k$ in a $(n+2)$-dimensional space-time.

To produce foliations we will use the fact that geodesics spheres of small radius around a point $p\in M$ form a foliation. This foliation can be perturbed in a suitable way. The perturbation procedure consist in a normal perturbation of the geodesics spheres and a perturbation of their center. The perturbation takes place in special coordinates around a point $p$ and an auxiliary manifold similar as in \cite{Ye}.

Denote by $R_p$ the injectivity radius of $(M,g)$ at $p$ and define $r_p:= \frac{1}{8} R_p$. Denote $\mathbb{B}_r:=\{x\in\mathbb{R}^{n+1}: \|x \|<r  \} $ where $\|\cdot \|$ is the Euclidean norm.  

For $\tau \in \mathbb{R}^{n+1} $ with $\| \tau\| < r_p$ we define $F_\tau : \mathbb{B}_{2r_p} \rightarrow M $ by
\begin{equation*}
  F_\tau(x)= \exp_{c(\tau)}(x^i e^\tau_i) 
\end{equation*}
where $c(\tau)= \exp_p (\tau^i e_i)$, $\{e_i\}$ is an orthonormal  Basis of $T_p M$, and $e_i^\tau$ their parallel transport to $c(\tau)$ along the geodesic $c(t \tau )_{0\leq t \leq 1
}$.  Consider also the dilation $\alpha_r(x)=rx $ for $r>0$. For each $\tau$ and $0<r<r_p$, the map $F_\tau \circ \alpha_r $ gives rise to rescaled normal coordinates centered at $c(\tau)$. The metric $g$ in this coordinates thus satisfies
\begin{equation*}
    g_{ij}= r^2 (\delta_{ij} + \sigma_{ij}(x r))
\end{equation*}
with $\sigma$ satisfying $|\sigma_{ij}(x)|\leq \|x\|^2$. We refer to this by the shorthand $g_{ij}= r^2 (\delta_{ij} + \mathcal{O}(\|xr\|^2 )) $. 

Consider the auxiliary manifold 
\begin{equation*}
  \left(\mathbb{B}_{2r_p}, g_{\tau,r}= r^{-2} \alpha_r^*(F_\tau^*(g)),k_{\tau, r}= r^{-1} \alpha_r^*(F_\tau^*(k)) \right).
\end{equation*}
Its metric is conformal to $g$ in the coordinates $F_\tau \circ \alpha_r$. For $r=0$,  $g_{\tau,0}$ is the Euclidean metric, and $k_{\tau,0}=0$. 

Let $\nu(x)$ be the outward unit normal to $\mathbb{S}^n$ and $\varphi \in \mathcal{C}^2 ( \mathbb{S}^n)$. Let $ S^n_\varphi :=\{x+ \varphi(x) \nu(x): x \in \mathbb{S}^n \}$ be the normal graph of $\varphi$ over $\IS^n$. There is $\delta_0>0$ such that if $\|\varphi \|_{\mathcal{C}^1}  <\delta_0$ then $S^n_\varphi$ is an embedded $\mathcal{C}^2$ surface in $\mathbb{R}^{n+1}$. Denote by $S_{\varphi,\tau,r}$ the surface $F_\tau(\alpha_r(S_\varphi^n))$. It is parameterized by 
\begin{equation}\label{parametrization}
    F_{\varphi,\tau,r}: \IS^n \to S_{\varphi,\tau,r} :x \mapsto F_\tau(r(x+\varphi(x)\nu(x))) = \exp_{c(\tau)}\big(rx(1+\varphi(x))\big)
\end{equation}

Following \cite{Ye}, define for $x\in\IS^n$
\begin{equation*}
    \begin{split}
        H(r,\tau, \varphi)(x) &= \textrm{ mean curvature of $S^n_\varphi $ at $x+ \varphi(x) \nu(x) $ with respect to $g_{\tau,r}$ on $\mathbb{B}_2 $}\\
        &=\textrm{ $r\cdot$[mean curvature of $S_{\varphi,\tau,r}$ at $F_{\varphi,\tau,r} (x)$ with respect to $g$]  }\\
        & =: r  H_g(r,\tau)(x)
    \end{split}
\end{equation*}
  where the second line comes from the scaling of the mean curvature. Similarly for $k^{\tau}:=(F_\tau^*(k))$ we have 
\begin{equation*}
    P(r,\tau, \varphi) 
    = \tr_{g_{\tau,r}} k_{\tau,r}- k_{\tau,r}(\nu_{\tau,r}, \nu_{\tau,r}) 
    = r (\tr_g k^\tau- k^\tau ( \nu,\nu))=:r P_g (r,\tau, \varphi),
\end{equation*}
where $\nu_{\tau,r}$ and $\nu$ are the normal to $S^n_\varphi $ with respect to $g_{\tau,r}$ and the normal to $S_{\varphi,\tau,r}$ with respect to $g$ respectively. The second equality comes from the scaling of $k_{\tau,r}$ and $ g_{\tau,r}$ and we denote   $P_g:= \frac{1}{r}P $.

The rescaling of $k_{\tau,r}$, is chosen such that, if for instance $S^n_\varphi $ is a surface with constant expansion equal to $n$ in $(\mathbb{B}_{2r_p}, g_{\tau,r})$ then it satisfies 
\begin{equation*}
    H_{\tau, r} \pm P_{\tau, r}= r(H_g \pm P_g)= n
\end{equation*}
and therefore the the surface $S_{\varphi,\tau,r}$ has constant expansion equal to $\frac{n}{r}$.

We introduce the following notation: The covariant derivatives will be denoted by $;$ and the partial derivatives $\frac{\partial}{\partial x^i}$  by a comma  or by $\partial_i$.  For any tensors $A$ and $B$  we write $A \ast B$ for any linear combination of contractions of $A$ and $B$ with the correspondent metric. 
\subsection{Sketch of the method}
The proof of our main results closely follows the strategy used in \cite{Ye}. In this section we give a brief sketch of the Lyapunov-Schmidt reduction process used there to find surfaces with constant mean curvature there. Assume that $p\in M $ is a non degenerate critical point of the scalar curvature that is, $\nabla \mathrm{Sc}_p=0$ and $\Hess\mathrm{Sc}_p$ is invertible. 

We cite the following lemma for future reference
\begin{lemma}[{\cite[Lemma 1.3]{Ye}}]\label{1.3}
Let
\begin{equation*}
  H_\varphi = \frac{\partial}{\partial \varphi} H(r,\tau,\varphi)_{|r=0, \varphi=0}, \quad  H_{\varphi r} = \frac{\partial}{\partial r} H_\varphi, 
    \quad\text{and}\quad
    H_{\varphi r r} = \frac{\partial^2}{\partial r^2} H_\varphi. 
\end{equation*}
Then for all $\tau \in \mathbb{R}^{n+1}$ we have
\begin{enum}
\item $H_{\varphi r}(0, \tau,0)=0$ 
\item $H_{\varphi rr}(0, \tau,0) $ is an even operator in the sense that if $\varphi$ is an even function $i.e.$ $\varphi(-x)=\varphi(x) $, then  $H_{\varphi rr}(0, \tau,0) \varphi$ is also an even function. 
\end{enum}
\end{lemma}

Calculating the mean curvature of a geodesic sphere yields the following expansion in $r$:
\begin{equation}
\label{H}
    H(r,\tau, 0)(x)= n-\frac{1}{3}  \mathrm{Ric}^{\tau}_{ij}(0) x^i x^j r^2- \frac{1}{4} \mathrm{Ric}^{\tau}_{ij;k}(0)x^ix^jx^k r^3 + \mathcal{O}(r^4).
\end{equation}
 In particular $H(0,\tau, 0)=n$. The idea is thus to use the implicit function theorem to find $r_0>0$, $\tau (r)$ and $\varphi(r)$ such that $H(r,\tau(r), \varphi(r))-n=0$ for all $r\in (0,r_0)$. To this end write $\mathcal{C}^{2, \frac{1}{2}}( \mathbb{S}^n) = K \oplus K^\bot$ where $K$ is the kernel of $-\Delta_{ \mathbb{S}^n} -n =\frac{\partial}{\partial \varphi} H(r,\tau,\varphi)|_{r=0, \varphi=0} :=L$ on Euclidean space. $K$ is the span of the first spherical harmonics. The equation $H(r,\tau(r), \varphi(r))-n=0$ is solved in two steps, first the projection to $K$ is made to vanish, then the projection to $K^\bot$. For both steps, the implicit function theorem is used in the following ways.

Let $\pi$ denote the orthogonal projection from $ \mathcal{C}^{0, \frac{1}{2}}( \mathbb{S}^n)$ onto $K$ and $T:K \rightarrow  \mathbb{R}^{n+1} $ the isomorphism sending $x^i_{| \mathbb{S}^n} $  (a spherical harmonic on the sphere) to the $i$th coordinate basis $e_i$. Let $ \Tilde{\pi}:= T \circ \pi$. After composition with a suitable scaling the equation projected to the kernel we find
\begin{equation*}
\left.\Tilde{\pi}\left(\frac{1}{r^3}( H(r,\tau, r^2 \varphi)-n)\right)\right|_{\substack{{r=0}\\{\tau=0}}}= -\frac{|\mathbb{S}^n|}{(n+1)(n+3)} \partial_i \mathrm{Sc}  \;e_i
\end{equation*}
which is equal to zero as $p$ is a critical point of the scalar curvature. Furthermore,
\begin{equation*}
\frac{\partial}{\partial \tau^k } \Tilde{\pi}\left.\left(\frac{1}{r^3} (H(r,\tau, r^2 \varphi)-n)\right)\right|_{\substack{{r=0}\\{\tau=0}}}= -\frac{|\mathbb{S}^n|}{(n+1)(n+3)} \partial_{\tau^k} \partial_i \mathrm{Sc} \; e_i.
\end{equation*}
Since $\Hess\mathrm{Sc}$ is invertible, by the implicit function theorem there is a function $\tau=\tau(r,\varphi)$ with $\tau(0,\varphi_0)=0$ and $\Tilde{\pi}(H(r, \tau, r^2\varphi))=0 $ around $r=0$, $\varphi=\varphi_0$. $\varphi_0$ does not play a role in the previous calculations and therefore can be fixed to be the function $\varphi_0 \in K^\bot$ determined by the equation $L\varphi_0 (x) = \frac{1}{3} \mathrm{R}^{\tau}_{ij}(0) x^i x^j  $ for $x\in\mathbb{S}^n $.

Consider the projection to $K^\bot$, again using a suitable rescaling to find
\begin{equation*}
\pi^\bot \left.\left(\frac{1}{r^2} (H(r,\tau, r^2 \varphi)-n)\right)\right|_{\substack{{r=0}\\{\tau=0}}}= L\varphi_0- \frac{1}{3} \mathrm{R}^{\tau}_{ij}(0) x^i x^j  =0
\end{equation*}
and 
\begin{equation*}
\frac{\partial}{\partial \varphi }\pi^\bot \left.\left(\frac{1}{r^2} (H(r,\tau, r^2 \varphi)-n)\right)\right|_{\substack{{r=0}\\{\tau=0}}}= L_{| K^\bot}  
\end{equation*}
By the implicit function theorem there exist $r_0>0$, $\varphi:[0,r_0) \to C^{2,\frac12}$ with $\varphi(0)=\varphi_0$ such that $H(r,\tau(r), \varphi(r))=n$ for all $r\in(0,r_0)$. This gives us a collection of CMC surfaces $\{S_r:=S_{\varphi(r),\tau(r),r}\mid r\in(0,r_0)\}$. A priori they are not necessarily a foliation. The $S_r$ are parameterized by (\ref{parametrization}) and therefore to find  the lapse function we need to calculate
\begin{equation*}
    \begin{aligned}
        \left.\frac{\partial F_{\varphi,\tau,r}}{\partial r}\right|_{r=0} &= (d_x \exp_{c(\tau(r))} ) \Big(x(1 +r^2 \varphi(r) + r(r^2 \varphi(r))_r )\Big)|_{r=0} \\
        & \qquad + \left( \frac{\partial  \exp_{c(\tau(r))}}{\partial r}  \right) ( rx(1 +r^2 \varphi(r)))|_{r=0}. 
    \end{aligned}
\end{equation*}
where we write  for simplicity $\varphi(r)=\varphi(r)(x)  $. This reduces to $ \frac{\partial F_{\varphi,\tau,r}}{\partial r}|_{r=0}= x + \frac{\partial \tau^k}{\partial r}|_{r=0} e_k  $. Hence the lapse function is given by 
\begin{equation}
\label{lapse}
    \alpha:=\left\langle \frac{\partial F_{\varphi,\tau,r}}{\partial r}_{|r=0},\nu \right\rangle = 1 +  \frac{\partial \tau^k}{\partial r}\langle   e_k, \nu \rangle. 
\end{equation}
Therefore, the $S_r$ form a foliation if $\alpha>0$. For the CMC surfaces this is the case, since \cite{Ye} showed that $\frac{\partial \tau^k}{\partial r}=0 $.
    
All of the basic structure of the proof shown before applies when trying to find a STCMC or constant expansion foliation. 
    
\section{Construction of the STCMC Foliation}
Since $k_{\tau,r} \rightarrow 0$ when $r \rightarrow 0$ the linearization of the STCMC operator satisfies that 
\begin{equation*}
    (H^2 -P^2)_{\varphi}(0,\tau, 0)= 2n(-\Delta_{ \mathbb{S}^n} -n)=2nL.
\end{equation*}
Therefore the kernel of this operator is the same as in the CMC case and thus it is generated by the first spherical harmonics 
$\{ x^i \mid 1\leq i \leq n+1\}$ where $x^i$ denotes the $i$-th coordinate function on $\mathbb{S}^n \subset\R^{n+1}$.
From \eqref{H} and the scaling of $P^2$ we get, recalling that $P=rP_g$:
\begin{equation}
\label{general0}
    \begin{aligned}
        (H^2-P^2)(r,\tau, 0)(x) 
        &= 
        n^2-\frac{2n}{3}  \mathrm{Ric}^{\tau}_{ij}(0) x^i x^j r^2- \frac{2n}{4} \mathrm{Ric}^{\tau}_{ij;k}(0)x^ix^jx^k r^3
        \\ &\qquad -r^2 P_g^2(r, \tau,0)  + \mathcal{O}(r^4).    
    \end{aligned}
\end{equation}
A slight adaption of \cite[Lemma 1.2]{Ye} gives
\begin{lemma}\label{1.2}
    The projection to the kernel of the space-time mean curvature (STCMC) is
    \begin{equation*}
        \begin{split}
            \Tilde{\pi}\left( (H^2 -P^2)(r,\tau,  \varphi)\right)_{|\substack{r=0,\tau=0}}&= -\frac{n|\mathbb{S}^n|}{(n+1)(n+3) } r^3 \mathrm{Sc}_{,l} e_l -  \Tilde{\pi}\left(P^2 (r,\tau, 0)\right)\\ &+\Tilde{\pi}\left(\int_0^1 (H^2 -P^2)_{\varphi}(r,\tau, t \varphi) dt \right) + \mathcal{O}(r^5)
        \end{split}
    \end{equation*}
\end{lemma}
\begin{proof}
Using that 
$$ (H^2-P^2)(r,\tau, 0)= n^2-\frac{2n}{3}  \mathrm{Ric}^{\tau}_{ij}(0) x^i x^j r^2- \frac{2n}{4} \mathrm{Ric}^{\tau}_{ij;k}(0)x^ix^jx^k r^3  -P^2(r,\tau, 0)+ \mathcal{O}(r^4),$$
and projecting $(H^2-P^2)(r,\tau, \varphi)$ to the kernel like it was done in \cite[Lemma 1.2]{Ye},  the result is obtained.
\end{proof}

\begin{theorem}\label{priSTCMC}
    Let  $p\in M$ be such that at $p$, $\mathrm{A_{ST}}=0$ and $\nabla \mathrm{A_{ST}} $ is invertible. Then there exist a positive constant $C$ depending on the dimension of $M$ with the following property:
    If at $p$ it holds that 
    \begin{equation*}
        C |(\nabla \mathrm{A_{ST}})^{-1}|\, (|k|^2 + |\mathrm{Ric}|  ) |k| |\nabla k|<1
    \end{equation*}
    then there exist $r_0>0$ and a smooth foliation $\mathcal{F}= \{ S_r : r\in (0, r_0) \} $ of STCMC spheres with $\sqrt{H_{S_r}^2-P_{S_r}^2}= \frac{n}{r}$. $\mathcal{F}$ is a foliation regularly centered at $p$ and its leaves $S_r$ are
    normal graphs over geodesic spheres of radius $r$.
\end{theorem}

\begin{proof}\  
\begin{steps}
\item[Construct a family of STCMC surfaces.] Start by expanding the operator
    \begin{equation}
\begin{split}
    (H^2 -P^2)(r, \tau, r^2 \varphi )=&  \int_0^1 \frac{\partial}{\partial t} ((H^2 -P^2)(r, \tau, t r^2 \varphi )) dt +(H^2 -P^2)(r, \tau, 0)  \\
    =& \int_0^1 \int_0^1 \frac{\partial}{\partial s} ((H^2 -P^2)_\varphi(sr, \tau, st r^2 \varphi )) ds r^2 \varphi dt + (H^2 -P^2)(r, \tau, 0)\\
    &+ (H^2 -P^2)_\varphi(0, \tau, 0)\varphi^2r^2
    \end{split}
\end{equation}
Then, by continuing the same procedure, one finds the following expression:
    \begin{equation}
    \label{genral}
        \begin{aligned}
            (H^2 -P^2)(r, \tau, r^2 \varphi ) &= (H^2 -P^2)(r, \tau, 0 ) +(H^2 -P^2)_\varphi (0, \tau, 0) \varphi r^2 \\
            &\quad+ (H^2 -P^2)_{r\varphi} (0,\tau,0 )\varphi r^3\\
            &\quad+ r^4 \int_0^1 \int_0^1 t (H^2 -P^2)_{\varphi \varphi}(sr, \tau,st r^2 \varphi )  \varphi \varphi \ds \dt\\
            &\quad+ r^4 \int_0^1 \int_0^1 \int_0^1 s (H^2 -P^2)_{\varphi r r}(usr, \tau,ust r^2 \varphi )   \varphi \du \ds \dt\\
            &\quad+ r^5 \int_0^1 \int_0^1 \int_0^1 s t (H^2 -P^2)_{\varphi \varphi r}(usr, \tau,ust r^2 \varphi )  \varphi \varphi \du \ds \dt
        \end{aligned}
    \end{equation}
    Where 
    \begin{equation*}
        (H^2 -P^2)_{\varphi \varphi}(r, \tau,  \phi )  \varphi \varphi' = \frac{d}{dt} (H^2 -P^2)_{\varphi }(r, \tau,  \phi +t \varphi')  \varphi \vert_{t=0}.
    \end{equation*}
    Note  also that  $$(H^2 -P^2)_\varphi (0, \tau, 0) \varphi= 2n (-\Delta_{ \mathbb{S}^n} -n)\varphi=2n L\varphi $$ and by Lemma \ref{1.3} and the scaling of $P(r,\tau, r^2 \varphi)$ we have 
    \begin{equation*}
        (H^2 -P^2)_{r\varphi} (0,\tau,0 ) =2 H_r(0,\tau, 0) H_\varphi(0,\tau, 0)
    \end{equation*}
    The proof of \cite[Lemma 1.3]{Ye} yields that $\frac{\partial}{\partial r} B(0, \tau, 0) =0 $ where $B$ is the second fundamental form of $S^n_\varphi $ with respect to $g_{\tau,r}$ on $\mathbb{B}_2 $. From this, it follows that $H_r(0,\tau,0)=0 $ so that altoghether $(H^2 -P^2)_{r\varphi} (0,\tau,0 ) =0$.
    
    To compute the projection to the kernel $K$, use Lemma \ref{1.2}, that $L$ is an even operator, and that $\pi(n^2)=0$. Hence,
    \begin{equation}\label{kern}
        \begin{aligned}
            \quad&\hspace{-4ex}\Tilde{\pi}\left(\frac{(H^2 -P^2)(r, \tau, r^2 \varphi )-n^2}{r^3}\right)\\
            &= -\frac{n|\mathbb{S}^n|}{(n+1)(n+3) } \mathrm{Sc}_{,l} e_l - \frac{1}{r^3} \Tilde{\pi}\left(P^2 (r,\tau, 0)\right) +\frac{2n}{r} \Tilde{\pi}(L\varphi)  \\
            &\quad+\Tilde{\pi}\Bigg(  r \int_0^1 \int_0^1 t (H^2 -P^2)_{\varphi \varphi}(sr, \tau,st r^2 \varphi )  \varphi \varphi \ds \dt\\
            &\quad+ r \int_0^1 \int_0^1 \int_0^1 s (H^2 -P^2)_{\varphi r r}(usr, \tau,ust r^2 \varphi )   \varphi \du \ds \dt\\
            &\quad+ r^2 \int_0^1 \int_0^1 \int_0^1 s t (H^2 -P^2)_{\varphi \varphi r} (usr, \tau,ust r^2 \varphi )  \varphi \varphi \du \ds \dt \Bigg)\\
            &\quad+ \mathcal{O}(r^2).
        \end{aligned}
    \end{equation}
    Recall that by our rescaling $\Tilde{\pi}\left(P^2(r,\tau,0 )\right)= r^2\Tilde{\pi}\left(P_g^2(r,\tau,0 )\right)$. Hence, for small $r>0$,
    \begin{equation*}
        \begin{aligned}        \Tilde{\pi}\left(P_g^2(r,\tau,0 )\right) 
            &= \Tilde{\pi}\left(P_g^2(0,\tau,0 )\right) + \frac{\partial}{\partial r }\Tilde{\pi}\left(P_g^2(r,\tau,0 )\right)_{| r=0} r \\
            &\quad 
            + \frac{1}{2} \frac{\partial^2 }{\partial r^2} \Tilde{\pi}\left(P_g^2(r,\tau,0 )\right)_{|\substack{r=0}} r^2 + \mathcal{O}(r^3).
        \end{aligned}
    \end{equation*}
    For $\tau=0$ and $r=0$ since $g^\tau_{ij}(rx)=\delta_{ij} + \mathcal{O}(r^2) $  it holds that $ g^\tau_{ij,l}(0)=0 $. This gives
    \begin{equation*}
        \begin{split}
            \Tilde{\pi}\left.\left( P_g^2(r,\tau,  0)\right)\right|_{r=0,\tau=0}
            &= \Bigg(\int_{\mathbb{S}^n} k_{nn}^\tau (rx) k_{jj}^\tau (rx) x^l - 2k_{nn}^\tau (rx) k_{ij}^\tau(rx) x^i x^j x^l\\ 
            &\quad + k_{ij}^\tau (rx) k_{pq}^\tau (rx) x^i x^j x^p x^q x^l d\mu\Bigg)\Big|_{r=0, \tau=0} \; e_l\\
            &=  k_{nn}^0 (0) k_{jj}^0 (0) \int_{\mathbb{S}^n}x^ld\mu \; e_l 
            - 2k_{nn}^0 (0) k_{ij}^0(0) \int_{\mathbb{S}^n}x^i x^j x^ld\mu\; e_l \\ 
            &\quad + k_{ij}^0(0) k_{pq}^0(0) \int_{\mathbb{S}^n}x^i x^j x^p x^q x^l d\mu \; e_l\\
            &= 0.
        \end{split}
    \end{equation*}
    Writing $k$ in place of $k^\tau(0)$ with $\tau=0$ and using that $\nu^\tau = x^i \frac{\partial}{\partial x^i} +  \mathcal{O}(r^2)  $, the radial derivative is given by
    \begin{equation*}
        \begin{split}
            \left.\frac{\partial}{\partial r}\Tilde{\pi}\left( P_g^2(r,\tau,  0)\right)\right|_{r=0,\tau=0} 
            &= 
            2k_{ii} k_{jj,k} \int_{\mathbb{S}^n}x^kx^ld\mu \; e_l 
            -2(k_{nn,k} k_{ij} + k_{nn} k_{ij,k}) \int_{\mathbb{S}^n}x^i x^jx^kx^l d\mu  \; e_l\\ 
            &\quad + 2k_{ij} k_{pq,k} \int_{\mathbb{S}^n}x^i x^j x^p x^px^kx^l d\mu \; e_l\\
            &=
            \frac{2|\mathbb{S}^n|}{(n+1)(n+3)(n+5) }\Big(   \partial_l(k_{ij}k_{ij} ) +4 \partial_i(k_{lj} k_{ji}) \\  
            &\quad+ ((n+1)(n+5)+1) \partial_l \left(\frac{k_{jj} k_{ii}}{2} \right) -2 (n+4) \partial_j(k_{lj} k_{ii})\Big) e_l\\
            &= 
            \frac{2|\mathbb{S}^n|}{(n+1)(n+3)(n+5) }\Big(    \partial_l(|k|^2 ) +4 \diver \left(  k^2 \right)_l \\    &+((n+1)(n+5)+1)\partial_l \Big(\frac{(\tr k)^2}{2}\Big)  - 2 (n+4) \diver \left(\tr k\cdot k\right)_l\Big) e_l.
        \end{split}
    \end{equation*}
    The derivatives with respect to $\tau$ are given by
    \begin{equation}
        \label{tau}
        \begin{split}
            &\left.\frac{\partial}{\partial \tau^\beta}\Tilde{\pi}( P_g^2(r,\tau,  0))\right|_{r=0,\tau=0}\\
            &\quad= 
            \Bigg(\frac{\partial}{\partial \tau^\beta}(k_{nn}^\tau (rx) k_{jj}^\tau (rx)) \int_{\mathbb{S}^n}x^ld\mu - \frac{\partial}{\partial \tau^\beta}(2k_{nn}^\tau (rx) k_{ij}^\tau(rx)) \int_{\mathbb{S}^n}x^i x^j x^ld\mu\\ 
            &\qquad+ \frac{\partial}{\partial \tau^\beta} (k_{ij}^\tau(rx) k_{pq}^\tau(rx) ) \int_{\mathbb{S}^n}x^i
            x^j x^p x^q x^l d\mu   \Bigg)\Bigg|_{r=0, \tau=0} e_l\\
            &\quad=0.
        \end{split}
    \end{equation}
    A computation similar to the one before gives
    \begin{equation*}
        \begin{split}
            &\left.\frac{\partial^2}{\partial r \partial \tau^\beta}\Tilde{\pi}( P_g^2(r,\tau,  0))\right|_{r=0,\tau=0}\\
            & \textcolor{white}{123456789101112131} = \frac{2|\mathbb{S}^n|}{(n+1)(n+3)(n+5) }\Big(   \partial_l\partial_{\tau^\beta } (k_{ij}k_{ij} ) +4 \partial_i\partial_{\tau^\beta }  (k_{lj} k_{ji}) \\  
            & \textcolor{white}{12345678910111213141}+ ((n+1)(n+5)+1) \partial_l \partial_{\tau^\beta } \left(\frac{k_{jj} k_{ii}}{2} \right) -2 (n+4) \partial_j \partial_{\tau^\beta }  (k_{lj} k_{ii})\Big) e_l\\
            &\textcolor{white}{123456789101112131}= 
            \frac{2|\mathbb{S}^n|}{(n+1)(n+3)(n+5) }\Big(    \partial_{\tau^\beta } \partial_l(|k|^2 ) +4\partial_{\tau^\beta }  \diver (  k^2 )_l \\    
            &\textcolor{white}{12345678910111213141}+((n+1)(n+5)+1)\partial_{\tau^\beta }  \partial_l \Big(\frac{(\tr k)^2}{2}\Big)  - 2 (n+4)\partial_{\tau^\beta }  \diver \left(\tr k \cdot k \right)_l \Big) e_l.\\
        \end{split}
    \end{equation*}
    Here we used in the last expression that partial derivatives commute when $r=0$.    

    Note that for any $\varphi_0 \in K^\bot$ one has $\Tilde{\pi}(L\varphi_0 )=0 $. Therefore for $\varphi_0 \in K^\bot$ to be fixed later we find
      \begin{equation*}
        \begin{split}
        &\left.\Tilde{\pi}\left(\frac{(H^2 -P^2)(r, \tau, r^2 \varphi_0 )-n^2}{r^3}\right)\right|_{r=0, \tau=0} \\
        &\textcolor{white}{123456789101112131415161718192}=   
        \frac{2|\mathbb{S}^n|}{(n+1)(n+3) }\Big( - \frac{n}{2} \mathrm{Sc}_{,l} -\frac{1}{(n+5)} \Big[   \partial_l(k_{ij}k_{ij} ) +4 \partial_i(k_{lj} k_{ji})\\
        &\textcolor{white}{12345678910111213141516171819202}+ ((n+1)(n+5)+1) \partial_l \left(\frac{k_{jj} k_{ii}}{2} \right) -2 (n+4) \partial_j(k_{lj} k_{ii}) \Big]\Big)e_l\\
        &\textcolor{white}{123456789101112131415161718192}= - \frac{2|\mathbb{S}^n|}{(n+1)(n+3) } \mathrm{A_{ST}}(p)_{ l}\; e_l =0
        \end{split}
    \end{equation*}
    and 
    \begin{equation*}
        \begin{split}
        &\left.\frac{\partial}{\partial \tau^\beta }\Tilde{\pi}\left(\frac{(H^2 -P^2)(r, \tau, r^2 \varphi_0 )-n^2}{r^3}\right)\right|_{r=0, \tau=0}\\
        &\quad=
        -\frac{2|\mathbb{S}^n|}{(n+1)(n+3)}\Big( \frac{n}{2} \partial_l\partial_{\tau^\beta } \mathrm{Sc} +\frac{1}{(n+5)} \Big[ \partial_l\partial_{\tau^\beta }(|k|^2 ) +4 \partial_i \partial_{\tau^\beta }  (  k^2)_i\\
        &\qquad\quad
        -2(n+4) \partial_i \partial_{\tau^\beta }  (\tr k\cdot k_{li}) + ((n+1)(n+5)+1)\partial_{\tau^\beta } \partial_l \Big(\frac{(\tr k)^2}{2}\Big)  \Big]\Big)\\
        &\quad= - \frac{2|\mathbb{S}^n|}{(n+1)(n+3) } \partial_{\tau^\beta }  \mathrm{\mathrm{A_{ST}}}(p)_{l}\; e_l.
        \end{split}
    \end{equation*}
    Since $\partial_{\tau^\beta } A_l$ is invertible, by the implicit function theorem there is a function $\tau =\tau(r, \varphi)$ with $\tau (0,\varphi_0)=0$  and $ \Tilde{\pi}\left((H^2 -P^2)(r, \tau, r^2 \varphi )\right) =0$ for $(r, \varphi)$ in a neighborhood of $(0, \varphi_0)$. We can fix $\varphi_0$ to be determined by $L\varphi_0 = \frac{1}{3}\mathrm{Ric}^0_{ij}(0) x^ix^j+ \frac{1}{2n}P^2_g(0,0,0)$. 

    This solves the equation projected to $K$ near $(r,\varphi)=(0,\varphi_0)$. Thus, it remains to solve the projection to $K^\bot$. From (\ref{general0})  we have that  (\ref{genral}) reduces to 
    \begin{equation}\label{1a}
        \begin{split}
            \frac{(H^2 -P^2)(r, \tau, r^2 \varphi )-n^2}{r^2} &=-\frac{2n}{3}  \mathrm{Ric}^{\tau}_{ij}(0) x^i x^j - \frac{2n}{4} \mathrm{Ric}^{\tau}_{ij;k}(0)x^ix^jx^k r- P_g^2(r, \tau,0)   +2 n L\varphi\\
            &+ r^2 \int_0^1 \int_0^1 t (H^2 -P^2)_{\varphi \varphi}(sr, \tau,st r^2 \varphi )  \varphi \varphi ds dt\\
            &+ r^2 \int_0^1 \int_0^1 \int_0^1 s (H^2 -P^2)_{\varphi r r}(usr, \tau,ust r^2 \varphi )   \varphi du ds dt\\
            &+ r^3 \int_0^1 \int_0^1 \int_0^1 s t (H^2 -P^2)_{\varphi \varphi r} (usr, \tau,ust r^2 \varphi )  \varphi \varphi du ds dt   + \mathcal{O}(r^2). \\
        \end{split}
    \end{equation}
    Therefore we find 
    \begin{equation*}
        \begin{split}
            \left.\left(\frac{(H^2 -P^2)(r, \tau, r^2 \varphi )-n^2}{r^2}\right)\right|_{r=0, \varphi=\varphi_0} &=  2n L\varphi_0 - \frac{2n}{3} \mathrm{Ric}^0_{ij}(0) x^ix^j- P^2_g(0,0,0)=0
        \end{split}
    \end{equation*}
    and 
    \begin{equation*}
        \begin{split}
            \left.\frac{\partial}{\partial \varphi} \left(\frac{(H^2 -P^2)(r, \tau, r^2 \varphi )-n^2}{r^2}\right)\right|_{r=0, \varphi=\varphi_0} &=  2n L\vert_{K^\bot}.
        \end{split}
    \end{equation*}
    Here $L$ is restricted to $K^\bot$ since we are considering the equation restricted to $K^\bot$. Hence, this the operator is invertible. By the implicit function theorem there exist $r_0>0$, $\tau:[0,r_0) \to \R^{n+1}$ and $\varphi:[0,r_0)\to C^{2,\frac12}(\IS^n)$ such that $(H^2 -P^2)(r, \tau(r), r^2 \varphi(r) )=n^2  $ for all $r\in(0,r_0)$. In particular, each $r\in(0,r_0)$ we obtain a STCMC surface $S_r:= S_{\varphi(r),\tau(r),r}$.
\item[The parameterization.]
    The parametrization of the family $\{S_r \mid r\in(0,r_0)\}$ is given by
    \eqref{parametrization}. Its lapse function is $\alpha= 1 +  \frac{\partial \tau}{\partial r}|_{r=0} \langle   e_k, \nu \rangle $ as in \eqref{lapse}. Again, this family forms a foliation if $|\frac{\partial\tau }{\partial r}|_{r=0}|<1 $.  To estimate $\frac{\partial\tau }{\partial r}|_{r=0} $ note that the equation $\Tilde{\pi}\left(\frac{(H^2 -P^2)(r, \tau, r^2 \varphi )}{r^3}\right)=0 $ implies that  
    \begin{equation*}
        \left.\frac{\partial}{\partial r} \Tilde{\pi}\left(\frac{(H^2 -P^2)(r, \tau, r^2 \varphi )}{r^3}\right)\right|_{r=0}=0.
    \end{equation*}
    By \eqref{kern} and chain rule this is equivalent to 
    \begin{equation}\label{estau}
        \begin{split}
            0 &= -\frac{n |\mathbb{S}^n|}{(n+1)(n+3)} \partial_{\tau^\beta } \partial_l \mathrm{Sc}^0(0) 
            \left.\frac{\partial \tau^\beta}{\partial r}\right|_{r=0} \; e_l
            - \left.\frac{\partial}{\partial r} \Tilde{\pi}\left( \frac{1}{r} P_g^2(r,\tau, 0) \right)\right|_{r=0}\\ 
            &\qquad + \frac{1}{2} \Tilde{\pi}\left( (H^2 -P^2)_{\varphi \varphi}(0, 0,0 ) \varphi_0 \varphi_0  \right)
            +\frac{1}{2} \Tilde{\pi}\left( (H^2 -P^2)_{\varphi r r}(0, 0,0 ) \varphi_0   \right).
        \end{split}
    \end{equation}
    The third term on the right satisfies
    \begin{equation*}
        \begin{split}
            \Tilde{\pi}\big( (H^2 -P^2)_{\varphi \varphi}(0, 0,0 ) \varphi_0 \varphi_0  \big)=  \Tilde{\pi}\big( 2 ((\Delta_{ \mathbb{S}^n} +n) (\Delta_{ \mathbb{S}^n} + n) + 2nH_{\varphi \varphi})  \varphi_0 \varphi_0  \big) =0
        \end{split}
    \end{equation*}
    since $P(0,0,0)=(rP_g(r,0,0 ))|_{r=0}=0$, and because $(\Delta_{ \mathbb{S}^n} +n) (\Delta_{ \mathbb{S}^n} + n)$ and $H_{\varphi \varphi}$ are even operators.  

    Using the chain rule, the fact that $\left.\frac{\partial^2}{\partial \tau^\alpha \partial \tau^\beta} \Tilde{\pi}\left( P_g^2(r,\tau(r), 0) \right)\right|_{r=0}=0$ and using that our derivatives commute for $r=0$ we find
    \begin{equation*}
        \begin{split}
            \left.\frac{\partial}{\partial r} \Tilde{\pi}\left( \frac{1}{r} P_g^2(r,\tau, 0) \right)\right|_{r=0}  = \frac{\partial^2}{\partial r \partial \tau^\beta} \Tilde{\pi}\left.\left( P_g^2(r,\tau, 0) \right)\right|_{r=0, \tau=0} \left.\frac{\partial \tau^\beta}{\partial r}\right|_{r=0}
        \end{split}
    \end{equation*}
    Inserting this into \eqref{estau} we find
    \begin{equation*}
         0 = 
         - \frac{2|\mathbb{S}^n|}{(n+1)(n+3) } \partial_{\tau^\beta } \mathrm{\mathrm{A_{ST}}}_l \left.\frac{\partial \tau^\beta}{\partial r}\right|_{r=0}\; e_l 
         +\frac{1}{2} \Tilde{\pi}\left( (H^2 -P^2)_{\varphi r r}(0, 0,0 ) \varphi_0  \right).
    \end{equation*}
    Therefore we have 
    \begin{equation*}
        \left.\frac{\partial \tau}{ \partial r}\right|_{r=0} 
        = 
        \left( \nabla \mathrm{A_{ST}} \right)^{-1} \frac{1}{2} \Tilde{\pi}\left( (H^2 -P^2)_{\varphi r r}(0, 0,0 ) \varphi_0  \right).
    \end{equation*}
    Using the linearization of $P^2$ found in \cite{STCMC}, that is 
    \begin{equation*}
        (P^2 )_{\varphi \varphi_0} = 2P \big( \left( \nabla_\nu \tr k - \nabla_\nu k(\nu, \nu)\right)\varphi_0 +2 k(\nabla \varphi_0, \nu) \big)
    \end{equation*}
    and that $\Tilde{\pi} \big((H^2)_{\varphi r r}(0, 0,0 ) \varphi_0 \big)=0 $ by Lemma \ref{1.3} we find
    \begin{equation}
        \begin{split}
            &\frac{1}{2} \Tilde{\pi}\left( (H^2 -P^2)_{\varphi r r}(0, 0,0 ) \varphi_0  \right)
            \\
            &\quad=
            \left.\Tilde{\pi}\Big( \left( \tr k -k(\nu,\nu) \right)
            \big( \left( \nabla_\nu \tr k - \nabla_\nu k(\nu, \nu)\right)\varphi_0 +2 k(\nabla \varphi_0, \nu) \big)\Big)\right|_{r=0}\\
            &\quad=
            \Bigg(\int_{\mathbb{S}^n} (\tr k  - k_{ij}\; x^i x^j  ) ( \partial_p \tr k  \; x^p - \partial_p k_{qk}  \;x^p x^q x^k ) \varphi_0 x^l\\
            &\qquad\qquad -2 \partial_\gamma \big( (\tr k  - k_{ij}\; x^i x^j  )k_{\gamma p} \; x^px^l \big)\varphi_0 d\mu\Bigg) \, e_l.
        \end{split}
    \end{equation}
    Since $\varphi_0$ is determined by $L\varphi_0 = \frac{1}{3} Ric^0_{ij}(0) x^ix^j+ \frac{1}{2n}P^2_g(0,0,0)$ one finds that  
    $\varphi_0 $ has the form 
    \begin{equation*}
        \varphi_0 =(k\ast k)_{ijpq} \;x^i x^j x^p x^q  + C\cdot ( \mathrm{Ric} +  k\ast k )_{ij} \; x^i x^j   + C \cdot (\mathrm{Sc} +  k\ast k )
    \end{equation*}
    where the $C$'s are constants depending on the dimension. From this it follows that 
    \begin{equation*}
        \begin{split}
            \frac{\partial \tau}{ \partial r}_{| r=0} = \left( \nabla \mathrm{A_{ST}} \right)^{-1} (k \ast \nabla k) (k\ast k + \mathrm{Ric}).
        \end{split}
    \end{equation*}
    Therefore 
    \begin{equation*}
        \left|\frac{\partial \tau}{ \partial r}\big|_{r=0}\right| 
        \leq 
        C |\left( \nabla \mathrm{A_{ST}} \right)^{-1}|\; |k|\; |\nabla k|\, (|k|^2 + |\mathrm{Ric}|).
    \end{equation*}
    Here $C$ only depends on $n$. Thus, if $ |\left( \nabla A \right)^{-1}| \; |k|\; |\nabla k| \, (|k|^2 + |\mathrm{Ric}|  ) $  is small enough, it follows that $|\frac{\partial \tau}{ \partial r}_{| r=0}| <1$ so that the $\{S_r\mid r\in (0,r_0)\}$ are a foliation, possibly after decreasing $r_0$ suitably.
\item[Centered parametrization:]
    Recall that the leaves of our foliation are given by normal graphs of geodesics spheres whose centers are perturbed via the parametrization (\ref{parametrization}):
    \begin{equation*}
         F_{\varphi,\tau,r}:\mathbb{R}^+ \times \mathbb{S}^n \to M : (r,x) \mapsto \exp_{c(\tau(r))}(rx(1+r^2 \varphi(r))).
    \end{equation*}
    In this step we show that there is a parametrization for which all these geodesic spheres are centered at $p$ instead of $c(\tau(r))$.
  
    Let
    \begin{equation*}
        \begin{split}
            \Psi(r,x) &:= \exp_p^{-1} \circ \exp_{c(\tau(r))}(x), \\
            \psi(r,x)&:=  \Psi(rx(1+r^2 \varphi (r)))\\
            \beta(r,x)&:= \frac{\psi(r,x)}{|\psi(r,x)|}.
        \end{split}
    \end{equation*}
    We now show that $\beta(r,\cdot)$ is a diffeomorphism for $r>0$ small enough. To this end, compute
    \begin{equation*}
        \begin{split}
            \frac{\partial  \psi}{\partial r} (x,0)
            &=\left.\left((d_x \Psi) (x +r^2 x\varphi(r) + r(r^2 x\varphi(r))_r) + \left( \frac{\partial  \Psi}{\partial r}  \right) ( r(x +r^2 x\varphi(r)))    \right)\right|_{r=0}\\
            &= x+ e_i   \left.\frac{\partial   \tau^i}{\partial r}\right|_{r=0}.  
        \end{split}
    \end{equation*}
    Hence it follows  that $  \psi(r,x) =r(x + \frac{\partial   \tau^i}{\partial r}_{|r=0} ) + \mathcal{O}(r^2)$. Since $\left|\frac{\partial   \tau}{\partial r}|_{r=0}\right| <1$ we have that for $r$ sufficiently small  $ \psi(r,x) \neq 0$ and thus
    \begin{equation*}
        \beta(r,x)
        = \frac{x + \frac{\partial   \tau }{\partial r}_{|r=0} + \mathcal{O}(r)}{\left|x + \frac{\partial   \tau}{\partial r}_{|r=0} + \mathcal{O}(r)\right|}
    \end{equation*}
    is a diffeomorphism.

    Since $\beta(r,x)$ is a diffeomorphism, for any $y \in \mathbb{S}^n$ there is an unique $x\in \mathbb{S}^n $  with $\beta(r,x)=y $, so that 
    \begin{equation*}
        \psi(r,x)= \beta(r,x) |\psi(r,x) |= y |\psi(r, \beta^{-1} (r,y)) |.
    \end{equation*}
    Using the shorthand $\Bar{\varphi}(r,y):= |\psi(r, \beta^{-1} (r,y)) |$ and recalling $\exp_p (\psi(r,x))= F_{\varphi,\tau,r}(r,x)$ this yields
    \begin{equation*}
        F_{\varphi,\tau,r}(r,x)= \exp_p (\psi(r,x) )= \exp_p (y \Bar{\varphi}(r,y) ).
    \end{equation*}
    Substituting $x=\beta^{-1} (r,y)$ gives the desired parameterization for the leaves of our foliation: 
    \begin{equation*} 
        \Bar{F}_{\varphi,\tau,r}(r,y):=F_{\varphi,\tau,r}(r, \beta^{-1} (r,y))= \exp_p (y \Bar{\varphi}(r,y) ).
    \end{equation*}
 
\item[Regularly centered:]
    To show that the foliation is regularly centered recall that the leaves of the foliation are normal graphs of the function $r^2 \varphi(r)$  over geodesics spheres of radius $r$. Therefore, the diameter of the leaves can be estimated as $ \diam S_r \leq Cr + r^2     \|\varphi\|_{\mathcal{C}^0} \leq C r$. Also, their second fundamental form can be estimated in terms of the second fundamental form of the geodesic sphere and $\Hess\varphi_0$. For geodesic spheres it holds that 
    \begin{equation*} 
      |B_{F_\tau(\mathbb{S}_r^n)}|^2= H^2+ |B_{F_\tau(\mathbb{S}_r^n)}- \tfrac{1}{n}\tr B_{F_\tau(\mathbb{S}_r^n)}  |^2 <Cr^{-2}.
    \end{equation*}
    The second estimate follows since for $r>0$ sufficiently small, the mean curvature dominates the trace free part. Using that $\| \varphi\|_{\mathcal{C}^2} < C$ with $C$ depending on the value of $\mathrm{Ric} $ and $k$ in these coordinates at $p$ we have
    \begin{equation*}
        |B_{S_r}|<C\big(|B_{F_\tau(\mathbb{S}_r^n)} | +r^2| \varphi|_{\mathcal{C}^2}\big) <C r^{-1}.
    \end{equation*}
    Hence, our foliation satisfies that  $ |B_{S_r}|\cdot \diam S_r < C  $ so that 
    \begin{equation*} 
        \sup_{S_r \in \mathcal{F} } \big( \sup_{S_r}|B_{S_r}|\cdot \diam S_r  \big)<\infty.
    \end{equation*}
    \end{steps}
\end{proof}
\begin{remark}\ 
    \begin{enum}
    \item The condition $C |(\nabla \mathrm{A_{ST}})^{-1}|\, (|k|^2 + |\mathrm{Ric}|  ) |k| |\nabla k|<1  $ is sufficient but not 
        necessary for the $S_r$ to be a foliation. The necessary condition is that $\alpha= 1 +  \frac{\partial \tau^k}{\partial r}_{|r=0} \langle   e_k, \nu \rangle >0$. If this condition does not hold, the proof of the theorem still guarantees that the  $S_r$'s are  a regularly centered concentration of STCMC surfaces at $p$.
    \item For $k=0$, we recover the CMC result of Ye \cite{Ye}.
    \end{enum}
\end{remark}

\section{Construction of the Constant Expansion Foliations}

Note that as $k_{\tau,r} \rightarrow 0$  when $r \rightarrow 0$ the CE stability operator when $r \rightarrow 0$ is given by the one  of the CMC operator:
\begin{equation*}
  (H\pm P )_{\varphi}(0,\tau, 0)= (-\Delta_{ \mathbb{S}^n} -n)=L   
\end{equation*}
From \eqref{H} and the scaling of $P$ it follows that
\begin{equation}
    \label{general00}
   (H\pm P)(r,\tau, 0)(x)= n-\frac{1}{3}  \mathrm{Ric}^{\tau}_{ij}(0) x^i x^j r^2- \frac{1}{4} \mathrm{Ric}^{\tau}_{ij;k}(0)x^ix^jx^k r^3 \pm r P_g(r, \tau,0)  + \mathcal{O}(r^4)  
\end{equation}
Thus \cite[Lemma 1.2]{Ye} directly gives 
\begin{lemma}
    The projection to the kernel of the expansions is
    \begin{equation*}
        \begin{split}
            \Tilde{\pi}\left( (H \pm P)(r,\tau,  \varphi)\right)_{|\substack{r=0,\tau=0}}&= -\frac{|\mathbb{S}^n|}{2(n+1)(n+3) } r^3 \,\mathrm{Sc}_{,l} e_l \pm  \Tilde{\pi}\left(P (r,\tau, 0)\right)\\ &+\Tilde{\pi}\left(\int_0^1 (H \pm P)_{\varphi}(r,\tau, t \varphi) dt \right) + \mathcal{O}(r^5)
        \end{split}
    \end{equation*}
\end{lemma}
The additional term is computed in the following lemma:
\begin{lemma}
    \label{deriv}
    It holds that
    \begin{equation}
    \begin{aligned}
        0 
        &=
        \Tilde{\pi}\left((P_g)(0,\tau,0) \right) 
        = \left.\frac{\partial^2}{\partial r^2}\Tilde{\pi}\left( P_g(r,\tau,  0)\right)\right|_{r=0,\tau=0}
        = \left.\frac{\partial}{\partial \tau^\beta}\Tilde{\pi}( P_g(r,\tau,  0))\right|_{r=0,\tau=0}\\
        &=\frac{\partial^3}{\partial \tau^\beta \partial r^2 }\Tilde{\pi}\left( P_g(r,\tau,  0)\right)_{|r=0,\tau=0}.
        \end{aligned}
    \end{equation}
    \begin{equation}
        \left.\frac{\partial}{\partial r}\Tilde{\pi}\left( P_g(r,\tau,  0)\right)\right|_{r=0,\tau=0} =\frac{|\mathbb{S}^n|}{(n+1) } E_l\, e_l
    \end{equation}
    \begin{equation}
        \left.\frac{\partial^2}{\partial r \partial \tau^\beta}\Tilde{\pi}( P_g(r,\tau,  0))\right|_{r=0,\tau=0} =\frac{|\mathbb{S}^n|}{(n+1) } \partial_{\tau^\beta} E_l\, e_l
    \end{equation}
\end{lemma}

\begin{proof}
    Using that $ \frac{\partial}{\partial r} g^\tau _{|r=0} =0$ it follows that
    \begin{equation*} 
        \Tilde{\pi}\left((P_g)(0,\tau,0) \right) = \int_{\mathbb{S}^n} k_{ii} x^l -k_{ij} x^i x^j x^l d\mu =0
    \end{equation*}
    \begin{equation*}
        \begin{aligned}
            \frac{\partial}{\partial r}\Tilde{\pi}\left( P_g(r,\tau,  0)\right)_{|r=0,\tau=0} =&  k_{jj,k} \int_{\mathbb{S}^n}x^kx^ld\mu \; e_l  - k_{ij,k} \int_{\mathbb{S}^n}x^i x^jx^kx^l d\mu  \; e_l\\
            =&\frac{|\mathbb{S}^n|}{(n+1) } \Big( \frac{n+2}{n+3} \partial_l \tr k -2 \partial_i k_{li}  \Big) e_l= \frac{|\mathbb{S}^n|}{(n+1) } E_l\, e_l,
        \end{aligned}
    \end{equation*}
    and
    \begin{equation*} 
        \begin{aligned}
            &\left.\frac{\partial^2}{\partial r^2}\Tilde{\pi}\left( P_g(r,\tau,  0)\right)\right|_{r=0,\tau=0} \\
            &\quad=
            k_{ij}\left.\frac{\partial^2 g^{\tau \,ij}}{\partial r^2}\right|_{r=0} \int_{\mathbb{S}^n}x^ld\mu \; e_l + k_{jj,kp} \int_{\mathbb{S}^n}x^kx^lx^pd\mu \; e_l  - k_{ij,k} \int_{\mathbb{S}^n}x^i x^jx^kx^lx^p d\mu  \; e_l =0.
        \end{aligned}
    \end{equation*}
    The derivatives with respect to $\tau$ are computed as in \eqref{tau}
    \begin{equation*} 
        \left.\frac{\partial}{\partial \tau^\beta}\Tilde{\pi}( P_g(r,\tau,  0))\right|_{r=0,\tau=0}=0.
    \end{equation*}
    Computations similar to the ones before yield
    \begin{equation*}
        \begin{split}
            \left.\frac{\partial^2}{\partial r \partial \tau^\beta}\Tilde{\pi}( P_g(r,\tau,  0))\right|_{r=0,\tau=0}
            &= 
            \frac{|\mathbb{S}^n|}{(n+1) } \left( \frac{n+2}{n+3} \partial_{\tau^\beta} \partial_l \tr k -2 \partial_{\tau^\beta} \partial_i k_{li}  \right) e_l =\frac{|\mathbb{S}^n|}{(n+1) } \partial_{\tau^\beta} E_l \,e_l.
        \end{split}
    \end{equation*}
\end{proof}
The following result follows completely analogous to the the previous section, with appropriate modifications.
\begin{theorem}
\label{priCE}
    Let  $p\in M$ be such that at $p$, $\mathrm{A_{CE}}=0$, $\nabla \mathrm{A_{CE}}$ is invertible and $k=0$. If $\frac{1}{(n+3)} |(\nabla \mathrm{A_{CE}})^{-1}|\, |   \mathrm{\hat{A}_{CE}}^\pm| <1  $   at $p$   there exist  $\delta^\pm>0$ and a smooth foliation $\mathcal{F}^\pm= \{ S_r^\pm : r\in (0, \delta^\pm)  \} $ of constant expansion  spheres with $H_{S_r^{\pm}}\pm P_{S_r^\pm}= \frac{n}{r}$. This foliation  $\mathcal{F}^\pm$ is regularly centered at $p$ and its leaves $S_r^\pm$ can be expressed as normal graphs over geodesic spheres of radius~$r$. 
\end{theorem}
\begin{proof}
    Similarly, as in \eqref{genral}, it holds that 
    \begin{equation}\label{genral2}
        \begin{split}
            (H \pm P)(r, \tau, r^2 \varphi )=& (H \pm P)(r, \tau, 0 ) +(H \pm P)_\varphi (0, \tau, 0) \varphi r^2 + (H \pm P)_{r\varphi} (0,\tau,0 )\varphi r^3\\
            &+ r^4 \int_0^1 \int_0^1 t (H \pm P)_{\varphi \varphi}(sr, \tau,st r^2 \varphi )  \varphi \varphi\ds\dt\\
            &+ r^4 \int_0^1 \int_0^1 \int_0^1 s (H \pm P)_{\varphi r r}(usr, \tau,ust r^2 \varphi )   \varphi \du \ds \dt\\
            &+ r^5 \int_0^1 \int_0^1 \int_0^1 s t (H\pm P)_{\varphi \varphi r}(usr, \tau,ust r^2 \varphi )  \varphi \varphi \du \ds \dt\\
        \end{split}
    \end{equation}
    At first, consider the projection onto $K$, dividing the equation by $r^2$. This gives that for arbitrary $\varphi_0 \in K^\bot$
    \begin{equation*}
        \label{ker}
        \begin{split}
            &\left.\Tilde{\pi}\left(\frac{(H \pm P)(r, \tau, r^2 \varphi )-n}{r^2}\right)\right|_{r=0,\tau=0 ,\varphi=\varphi_0} \\
            &\quad=
            \left.\pm  \Tilde{\pi}\left( \frac{1}{r^2} P(r,\tau,0) \right)\right|_{r=0,\tau=0}
            = \left. \pm  \Tilde{\pi}\left( \frac{1}{r} P_g(r,\tau,0) \right)\right|_{r=0,\tau=0}
            = \pm \frac{|\mathbb{S}^n|}{(n+1)} \mathrm{A_{CE}}_{l}\, e_l =0
        \end{split}
    \end{equation*}
    and 
    \begin{equation*}
        \begin{split}
            &\left.\frac{\partial}{\partial \tau^\beta }\Tilde{\pi}\left(\frac{(H \pm P)(r, \tau, r^2 \varphi )-n}{r^2}\right)\right|_{r=0, \tau=0, \varphi=\varphi_0 }
            \\ 
            &\quad= 
            \left.\pm \frac{\partial}{\partial \tau^\beta} \Tilde{\pi}\left( \frac{1}{r} P_g(r,\tau,0) \right)\right|_{r=0,\tau=0}
            = \left.\pm \frac{\partial^2}{\partial  \tau^\beta \partial r} \Tilde{\pi}\left( \frac{1}{r} P_g(r,\tau,0) \right)\right|_{r=0,\tau=0}
            = \pm \frac{|\mathbb{S}^n|}{(n+1)} \partial_{\tau^\beta} \mathrm{A_{CE}}_l\, e_l. 
        \end{split}
    \end{equation*}
    By assumption $\partial_{\tau^\beta} \mathrm{A_{CE}}_l$ is invertible. Hence, by the implicit function theorem, there is a function $\tau =\tau(r, \varphi)$ with $\tau (0,\varphi_0)=0$  and $ \Tilde{\pi}\left((H \pm P)(r, \tau, r^2 \varphi )\right) =0$ for $(r, \varphi)$ in a neigborhood of  $(0, \varphi_0)$. We choose $\varphi_0$ to be determined by $L\varphi_0 =\pi^\bot ( \frac{1}{3} \mathrm{Ric}^0_{ij}(0) x^ix^j \mp  \frac{\partial}{\partial r} \left(  P_g(r,\tau,0) \right)_{|r=0,\varphi_0=0} )$.

    It remains to consider the projection to $K^\bot$. From \eqref{genral2} and \eqref{general00} we have 
    \begin{equation*}
        \begin{split}
            \pi^\bot \left(\frac{(H \pm P)(r, \tau, r^2 \varphi )-n}{r^2}\right)_{|r=0, \varphi=\varphi_0} &=   L\varphi_0- \pi^\bot(  \frac{1}{3} \mathrm{Ric}^0_{ij}(0) x^ix^j \mp \frac{\partial}{\partial r} \left(  P_g(r,\tau,0) \right)_{|r=0,\tau=0} )=0
        \end{split}
    \end{equation*}
    where we used that $P(0,0,0)=0$. Hence,
    \begin{equation*}
        \begin{split}
            \frac{\partial}{\partial \varphi} \left(\frac{(H \pm P)(r, \tau, r^2 \varphi )-n}{r^2}\right)_{|r=0, \varphi=\varphi_0} &=   L\vert_{K^\bot}.
        \end{split}
    \end{equation*}
    Here $L$ is restricted to $K^\bot$ since our equation was restricted to $K^\bot$. Therefore the operator is invertible and thus, by the implicit function theorem, there exist $r_0>0$, $\tau:(0,r_0)\to M$, and $\varphi:(0,r_0)\to C^{2,\alpha}(S^2)$ such that $(H\pm P)(r, \tau(r), r^2 \varphi(r) )=n $ for all $r\in(0,r_0)$. This means that for each $r$ we obtain a constant expansion surface $S_r:= S_{r,\tau(r),\varphi(r)}$.
    
    To show that the $\{S_r\mid r\in(0,r_0)\}$ form a foliation, we estimate $\frac{\partial \tau^\beta}{ \partial r }$ as in the proof of Theorem \ref{priSTCMC}. To this end compute 
    \begin{equation}
    \label{cefoli}
    \begin{split}
        0=&\frac{\partial}{\partial r} \Tilde{\pi}\left.\left(\frac{(H \pm P)(r, \tau(r), r^2 \varphi(r) )-n}{r^2}\right)\right|_{r=0}\\
        =&-\frac{|\mathbb{S}^n|}{2(n+1)(n+3) }  \,\mathrm{Sc}_{,l} e_l + \Tilde{\pi}\left( (H \pm P)_{r\varphi} (0,\tau,0 )\varphi_0 \right) \pm \frac{\partial}{\partial r} \left.\left(\frac{1}{r}  \Tilde{\pi}\left(P_g (r,\tau, 0)\right)\right)\right|_{r=0, \tau=0}  \\
        =&-\frac{|\mathbb{S}^n|}{2(n+1)(n+3) }  \,\mathrm{Sc}_{,l} e_l + \Tilde{\pi}\left( (H \pm P)_{r\varphi} (0,\tau,0 )\varphi_0 \right) \mp \left.\left( \frac{1}{r^2} \Tilde{\pi}\left(P_g (r,\tau, 0)\right) \right)\right|_{r=0} \\
        &\pm \left.\left(\frac{1}{r} \frac{\partial}{\partial r} \Tilde{\pi}\left(P_g (r,\tau, 0)\right)\right)\right|_{r=0, \tau=0}  \pm \left.\left(\frac{1}{r} \frac{\partial}{\partial \tau^\beta} \Tilde{\pi}\left(P_g (r,\tau, 0)\right)\right)\right|_{r=0, \tau=0}  \left.\frac{\partial \tau^\beta}{ \partial r }\right|_{r=0}\\
        =& -\frac{|\mathbb{S}^n|}{2(n+1)(n+3) }  \,\mathrm{Sc}_{,l} e_l + \Tilde{\pi}\left( (H \pm P)_{r\varphi} (0,\tau,0 )\varphi_0 \right) \pm  \frac{\partial ^2}{\partial r \partial \tau^\beta} \Tilde{\pi}\left.\left(P (r,\tau, 0)\right)\right|_{r=0, \tau=0}  \cdot\left.\frac{\partial \tau^\beta}{ \partial r }\right|_{r=0}\\
        =&  -\frac{|\mathbb{S}^n|}{2(n+1)(n+3) }  \,\mathrm{Sc}_{,l} e_l + \Tilde{\pi}\left( (H \pm P)_{r\varphi} (0,\tau,0 )\varphi_0 \right) \pm \frac{|\mathbb{S}^n|}{(n+1) }  \partial_{\tau^\beta} \mathrm{A_{CE}}_l \, e_l \left.\frac{\partial \tau^\beta}{ \partial r }\right|_{r=0}
    \end{split}
    \end{equation}
    Consider the second term in the last line. Recall that $\varphi_0$ is determined by the equation  
    \begin{equation}
        \begin{split} 
        L\varphi_0 
        &=\pi^\bot \left( \frac{1}{3} Ric^0_{ij}(0) x^ix^j \mp  \frac{\partial}{\partial r} \left.\left(  P_g(r,\tau,0) \right)\right|_{r=0} \right)\
        \\
        &= \frac{1}{3} \mathrm{Ric}_{ij} x^ix^j \pm k_{ij,p} x^i x^jx^p \mp \frac{1}{n+3} (\partial_i \tr k + 2 k_{ir,r})x^i.
        \end{split}
    \end{equation}
    The solution to this equation in $K^\bot$ is given by
    \begin{equation}
    \label{varphi}
        \varphi_0 = \pm \frac{1}{2(n+3)}k_{ij,p} x^i x^j x^p + \frac{1}{3(n+2)} \mathrm{Ric}_{ij} x^i x^j \mp \frac{1}{2(n+3)^2} (\partial_i \tr k + 2 k_{ir,r})x^i  - \frac{2}{3(n+2) }    \mathrm{Sc}
    \end{equation}
    Using the linearization   $(P )_\varphi \varphi_0 =  \big( \left( \nabla_\nu \tr k - \nabla_\nu k(\nu, \nu)\right)\varphi_0 +2 k(\nabla \varphi_0, \nu) \big)$ and $k=0$ at $p$ we find after a long calculation (in which we also use that as $E=0$, $k_{lr,r} = \frac{n+2}{2(n+3)} \partial_l \tr k$) that
    \begin{equation}\label{proj0}
        \begin{split}
        \Tilde{\pi}\left( (H \pm P)_{r\varphi} (0,\tau,0 )\varphi_0 \right) =&  \pm\int_{\mathbb{S}^n}(  k_{ii,j} x^jx^l - k_{ij,p} x^i x^j x^p x^l ) \varphi_0  d\mu \, e_l \\
       =&\pm  \frac{|\mathbb{S}^n|}{3(n+1)(n+3)^2(n+5)}\Big( \frac{2(n^2 +6n +10)}{(n+3)} \mathrm{Ric}_{lr}\, \partial_r \tr k  - 4\mathrm{Ric}_{rs}\,  k_{ls,r}\\
        &-2 \mathrm{Ric}_{rs}\,  k_{rs,l}  - \frac{n^3 +14n^2 +52n +60}{n(n+3)} \mathrm{Sc}\, \partial_l  \tr k \Big)e_l.
        \end{split}
    \end{equation}  
    Inserting this into \eqref{cefoli}  we find 
    \begin{equation*}
        \begin{split}
        0
        &= -\frac{|\mathbb{S}^n|}{2(n+1)(n+3) }  \,\mathrm{Sc}_{,l} e_l \pm  \frac{|\mathbb{S}^n|}{3(n+1)(n+3)^2(n+5)}\Big( \frac{2(n^2 +6n +10)}{(n+3)} \mathrm{Ric}_{lr}\, \partial_r \tr k  \\
        &\qquad - 4\mathrm{Ric}_{rs}\,  k_{ls,r} -2 \mathrm{Ric}_{rs}\,  k_{rs,l}  - \frac{n^3 +14n^2 +52n +60}{n(n+3)} \mathrm{Sc}\, \partial_l  \tr k \Big) e_l  \pm \frac{|\mathbb{S}^n|}{(n+1) }  \partial_{\tau^\beta} E_l \, e_l\left.\frac{\partial \tau^\beta}{ \partial r }\right|_{r=0}\\
        &=
        \frac{|\mathbb{S}^n|}{(n+1)(n+3) }  (\mathrm{\hat{A}_{CE}}^{\pm})_l \, e_l   
        \pm \frac{|\mathbb{S}^n|}{(n+1) }  \partial_{\tau^\beta} (\mathrm{A_{CE}})_l \, e_l
        \left.\frac{\partial \tau^\beta}{ \partial r }\right|_{r=0}.
        \end{split}
    \end{equation*}
    Thus $\left.\frac{\partial \tau}{ \partial r }\right|_{r=0} =\frac{1}{(n+3)} (\nabla \mathrm{A_{CE}})^{-1} \,    \mathrm{\hat{A}_{CE}}^\pm $. Since by assumption $\left|\left.\frac{\partial \tau}{ \partial r }\right|_{r=0}\right| \leq \frac{1}{(n+3)} |(\nabla \mathrm{A_{CE}})^{-1}|\, |       \mathrm{\hat{A}_{CE}}^\pm| <1 $  we have a foliation.    

    The rest of the proof follows as in Theomrem \ref{priSTCMC}.
\end{proof}

\begin{remark}
  Note that in contrast to Theorem \ref{priSTCMC} in this Theorem we can not allow $k$ to be zero in a neighborhood of $p$ since then $\nabla \mathrm{A_{CE}}$ would not be invertible. Therefore it is not a generalization of the CMC foliation of \cite{Ye}. However, using different assumptions than Theorem~\ref{priCE}, the following result in fact generalizes the CMC foliation.
\end{remark}
\begin{theorem}
    \label{secCMC}
    Let  $p\in M$ be such  that at $p$, $\mathrm{A_{CE}}=  \mathrm{\hat{A}_{CE}}^\pm =0$, $k= \nabla \mathrm{A_{CE}}=0$, $ \Hess\mathrm{A_{CE}}=0 $ and $\nabla   \mathrm{\hat{A}_{CE}}^\pm  + \hat{T}$  is invertible. Then there exist a constant $C$ depending on the dimension of $M$ such that if at $p$, $ C |(\nabla   \mathrm{\hat{A}_{CE}}^\pm  + \hat{T})^{-1} |\left( |\nabla k|\,( |\mathrm{Ric}|+|\nabla k|  +|\nabla \nabla k|) +|\nabla \nabla \nabla k| \right)  <1  $ then there exist  $\delta^\pm>0$ and a smooth foliation $\mathcal{F}^\pm= \{ S_r^\pm : r\in (0, \delta^\pm)  \} $ of constant expansion  spheres with $H_{S_r^\pm}\pm P_{S_r^\pm}= \frac{n}{r}$. $\mathcal{F}^\pm$ is a foliation regularly centered at $p$ and each of its leaves $S_r^{\pm}$ is a normal graph over a geodesic sphere of radius $r$. 
\end{theorem}
\begin{proof}
We proceed as in the proof of Theorem \ref{priCE} but instead of dividing by $r^2$ when projecting to $K$ we divide by $r^3$. As in the proof of Theorem \ref{priCE} we consider $\varphi_0$ to be given by (\ref{varphi}), then using (\ref{proj0}) we find
 \begin{equation*}
     \begin{split}
        &\left.\Tilde{\pi}\left(\frac{(H \pm P)(r, \tau, r^2 \varphi )-n}{r^3}\right)\right|_{r=0, \tau=0, \varphi= \varphi_0}
        \\
        &\quad = -\frac{|\mathbb{S}^n|}{2(n+1)(n+3)} \partial_l \mathrm{Sc}  \;e_l  \pm  \Tilde{\pi}\left.\left(\frac{1}{r^2} P_g (r,\tau, 0)\right)\right|_{r=0,\tau=0} + \Tilde{\pi}\left( (H \pm P)_{r\varphi} (0,0,0 )\varphi_0 \right)\\
        &\quad =\frac{|\mathbb{S}^n|}{(n+1)(n+3) }   (\mathrm{\hat{A}_{CE}}^\pm)_l \, e_l  =0
     \end{split}
 \end{equation*}
 where we used that $\Tilde{\pi}\left.\left(\frac{1}{r^2} P_g (r,\tau, 0)\right)\right|_{r=0,\tau=0}=0$ since $E=0$. For the derivative we have 
 \begin{equation}\label{proj2}
    \begin{split}
        &\frac{\partial}{\partial \tau^\beta}\Tilde{\pi}\left.\left(\frac{(H \pm P)(r, \tau, r^2 \varphi )-n}{r^3}\right)\right|_{r=0, \tau=0, \varphi= \varphi_0}\\
        &\textcolor{white}{123456789101112131415161718}= -\frac{|\mathbb{S}^n|}{2(n+1)(n+3)} \partial_{\tau^\beta} \partial_l \mathrm{Sc}  \;e_l 
        \pm \underbrace{\frac{\partial}{\partial \tau^\beta}  \Tilde{\pi}\left.\left(\frac{1}{r^2} P_g (r,\tau, 0)\right)\right|_{r=0,\tau=0}}_{=0}
        \\
        &\textcolor{white}{12345678910111213141516171819}
        + \frac{\partial}{\partial \tau^\beta}  \Tilde{\pi}\left.\left( (H \pm P)_{r\varphi} (0,\tau,0 )\varphi_0 \right)\right|_{\tau=0, \varphi= \varphi_0}.
     \end{split}
 \end{equation}
 Consider the third term on the right hand side and compute
  \begin{equation}
  \label{23}
     \begin{split}
        &\frac{\partial}{\partial \tau^\beta}  \Tilde{\pi}\left.\left( (H \pm P)_{r\varphi} (0,\tau,0 )\varphi_0 \right)\right|_{ \tau=0, \varphi= \varphi_0} 
        \\
        &\quad= \pm \frac{\partial}{\partial \tau^\beta} \left.\left( \int_{\mathbb{S}^n}(  k^\tau_{ii,j} x^jx^l - k^\tau_{ij,p} x^i x^j x^p x^l ) \varphi  +2k^\tau_{ij} \nabla^{\mathbb{S}^n}_i \varphi x^j x^l  d\mu \right)\right|_{\tau=0, \varphi =\varphi_0}  e_l\\
        &\quad=  \pm  \left( \int_{\mathbb{S}^n}(  \partial_{\tau^\beta} k_{ii,j} x^jx^l - \partial_{\tau^\beta} k_{ij,p} x^i x^j x^p x^l ) \varphi_0
            +2\partial_{\tau^\beta}k_{ij} \nabla^{\mathbb{S}^n}_i \varphi_0 x^j x^l  d\mu \right) \, e_l\\ 
        &\quad=\pm  \Big( \int_{\mathbb{S}^n}(  \partial_{\tau^\beta} k_{ii,j} x^jx^l - \partial_{\tau^\beta}k_{ij,p} x^i x^j x^p x^l ) \varphi_0
        -2\partial_{\tau^\beta} k_{\alpha j}  \nabla_\alpha^{\mathbb{S}^n} ( x^j x^l) \varphi_0   d\mu \Big) \, e_l.
     \end{split}
 \end{equation}
The first two terms are computed in \eqref{proj0}. For the remaining term note that $\partial_{\tau^\beta} k_{\alpha j}  \nabla_\alpha^{\mathbb{S}^n} x^s= \partial_{\tau^\beta} k_{s j} - \partial_{\tau^\beta}k_{r j} x^s x^r$ (where we denote the tangential vectors by $\alpha$ i.e. $e_\alpha$ tangent to $\mathbb{S}^n$). Taking also into account\eqref{varphi} we have that \eqref{23} reduces to 
\begin{equation}
     \begin{aligned}
         \frac{|\mathbb{S}^n|}{(n+1)(n+3) } & \Bigg( \partial_{\tau^\beta} (   (\mathrm{\hat{A}_{CE}}^\pm)_l + \frac{1}{2}\partial_{\tau^\beta} \partial_l \mathrm{Sc}  ) 
         +\frac{2}{(n+1)(n+3)} (\partial_{\tau^\beta} \tr k\, \partial_l \tr k +2 \partial_{\tau^\beta} \tr k \, k_{lr,r} \\
         &\qquad \qquad +2 \partial_i \tr k\ \partial_{\tau^\beta} k_{li} 
         +4 \partial_{\tau^\beta} k_{ij}\, k_{il,j} + 4 \partial_{\tau^\beta} k_{lr}\, k_{rs,s} +2 \partial_{\tau^\beta}k_{ij }\, k_{ir,r})\\
         &-\frac{2}{(n+3)^2}( \partial_{\tau^\beta} \tr k\, \partial_l \tr k +2 \partial_{\tau^\beta} \tr k \, k_{lr,r} + 4 \partial_{\tau^\beta} k_{lr}\, k_{rs,s}
         +2 \partial_i \tr k\  \partial_{\tau^\beta}k_{li} ) \Bigg)e_l\\
         &\hspace{-10ex}= 
         \frac{|\mathbb{S}^n|}{(n+1)(n+3) } \Bigg( \partial_{\tau^\beta} (   (\mathrm{\hat{A}_{CE}}^\pm)_l +\frac{1}{2} \partial_{\tau^\beta}\partial_l \mathrm{Sc}  ) \,+ \frac{4}{(n+3)(n+5)} \big((2 \partial_{\tau^\beta} k_{ij} \, k_{il, j} + \partial_{\tau^\beta} k_{ij}\, k_{ij,l})
         \\ & \qquad\qquad
         -\frac{2n+5}{(n+3)^2} (\partial_{\tau^\beta} \tr k\, \partial_l \tr k + 2 \partial_i \tr k \,\partial_{\tau^\beta} k_{il} )  \big)   \bigg)e_l\\
         &\hspace{-10ex}=
         \frac{|\mathbb{S}^n|}{(n+1)(n+3) } \Big( \partial_{\tau^\beta} (   (\mathrm{\hat{A}_{CE}}^\pm)_l + \frac{1}{2} \partial_{\tau^\beta} \partial_l \mathrm{Sc}  ) + \hat{T}_{l\tau^\beta}   \Big)e_l.
     \end{aligned}
 \end{equation}
 Here we used at the end that since $\mathrm{A_{CE}}=0$ also $k_{lr,r} = \frac{n+2}{2(n+3)} \partial_l \tr k$. Substituting this back into \eqref{proj2} we obtain 
\begin{equation}
    \frac{\partial}{\partial \tau^\beta}\Tilde{\pi}\left.\left(\frac{(H \pm P)(r, \tau, r^2 \varphi )-n}{r^3}\right)\right|_{r=0, \tau=0, \varphi= \varphi_0}= \frac{|\mathbb{S}^n|}{(n+1)(n+3) } \big( \partial_{\tau^\beta}    (\mathrm{\hat{A}_{CE}}^\pm)_l  + \hat{T}_{l \tau^\beta}   \big)e_l.
\end{equation}
Since $(\mathrm{\hat{A}_{CE}}^\pm)_l  + \hat{T}_{l \tau^\beta}$ is invertible at $p$, we can apply the implicit function theorem to find $\tau =\tau(r, \varphi)$ with $\tau (0,\varphi_0)=0$ as in the previous proofs. The projection to $K^\bot$ is done in exactly the same way as in Theorem \ref{priCE} therefore obtaining the required $\tau(r)$ and $\varphi(r)$.
 
 To see that these surfaces form a foliation we use
  \begin{equation*}
    \begin{split}
    0
    &=
    \frac{\partial}{\partial r} \Tilde{\pi}\left.\left(\frac{(H \pm P)(r, \tau(r), r^2 \varphi(r) )-n}{r^3}\right)\right|_{r=0}
    \\ &=
    -\frac{|\mathbb{S}^n|}{2(n+1)(n+3) }  \,\partial_{\tau^\beta} \partial_l\mathrm{Sc} \left.\frac{\partial \tau^\beta}{ \partial r }\right|_{r=0}  e_l +\frac{\partial}{\partial \tau^\beta}  \Tilde{\pi}\left.\left( (H \pm P)_{r\varphi} (0,\tau,0 )\varphi_0 \right)\right|_{\tau=0} \left.\frac{\partial \tau^\beta}{\partial r}\right|_{r=0} 
    \\ &\quad
    \pm \frac{\partial}{\partial r} \left.\left( \frac{1}{r^2} \Tilde{\pi}\left(P_g (r,\tau, 0)\right) \right)\right|_{r=0} 
    + \frac{1}{2} \Tilde{\pi}\left.\left( (H \pm P)_{\varphi rr} (0,\tau,0 )\varphi_0 \right)\right|_{\tau=0}
    \\ &= 
    \frac{|\mathbb{S}^n|}{(n+1)(n+3) } \big( \partial_{\tau^\beta}    (\mathrm{\hat{A}_{CE}}^\pm)_l  + \hat{T}_{l\tau^\beta}   \big)
    \left.\frac{\partial \tau^\beta}{ \partial r }\right|_{r=0}  e_l  \pm \frac{\partial}{\partial r} \left.\left( \frac{1}{r^2} \Tilde{\pi}\left(P_g (r,\tau, 0)\right) \right)\right|_{r=0}
    \\ &\quad
    + \frac{1}{2} \Tilde{\pi}\left.\left( (H \pm P)_{\varphi rr} (0,\tau,0 )\varphi_0 \right)\right|_{\tau=0}.
    \end{split}
\end{equation*}
This gives
\begin{equation*}
    \begin{split}
        \left.\frac{\partial \tau}{ \partial r }\right|_{r=0}  
        &=
        \frac{(n+1)(n+3) }{|\mathbb{S}^n|}(\nabla   \mathrm{\hat{A}_{CE}}^\pm  + \hat{T})^{-1}
        \Bigg(  \frac{1}{2} \Tilde{\pi}\left.\left( (H \pm P)_{\varphi rr} (0,\tau,0 )\varphi_0 \right)\right|_{\tau=0}\\
        &\quad \pm \frac{\partial}{\partial r} \left.\left( \frac{1}{r^2} \Tilde{\pi}\left(P_g (r,\tau, 0)\right) \right)\right|_{r=0} \Bigg).
    \end{split}
\end{equation*}
Note that by Lemma \ref{1.3} $ (H )_{\varphi rr} (0,\tau,0 )\varphi$ is an even operator and that the odd part of $\varphi_0$ is $\varphi_0^{odd}= \pm \frac{1}{2(n+3)}k_{ij,p} x^i x^j x^p  \mp \frac{1}{2(n+3)^2} (\partial_i \tr k + 2 k_{ir,r})x^i$. Thus we can estimate    
\begin{equation*}
    \big|\Tilde{\pi}\left.\left( (H \pm P )_{\varphi rr} (0,\tau,0 )\varphi_0 \right)\right|_{\tau=0}\big| 
    \leq C |\nabla k|\,( |\mathrm{Ric}|+|\nabla k|  +|\nabla \nabla k|).
\end{equation*}
Here $C$ is a constant depending only on $n$. Since $\Hess\mathrm{A_{CE}}=0$ it follows that the term $\frac{\partial}{\partial r} \left.\left( \frac{1}{r^2} \Tilde{\pi}\left(P_g (r,\tau, 0)\right) \right)\right|_{r=0}$ is given by a linear combination of contractions of $\nabla \nabla \nabla k$. Therefore this term can be estimated by a constant depending on $n$ times $|\nabla \nabla \nabla k|$. It follows that 
\begin{equation*}
    \left|\left.\frac{\partial \tau}{ \partial r }\right|_{r=0} \right|  
    \leq 
    C |(\nabla   \mathrm{\hat{A}_{CE}}^\pm  + \hat{T})^{-1} |\,\left( |\nabla k|\,( |\mathrm{Ric}|+|\nabla k|  +|\nabla \nabla k|) +|\nabla \nabla \nabla k| \right).
\end{equation*}
By assumption the right hand side is less than $1$ so that we have a foliation. The rest of the proof follows as in Theorem~\ref{priSTCMC}. 
\end{proof}
\begin{remark}\ 
\begin{enum}
\item As for the STCMC case, the assumption $\frac{1}{(n+3)} |(\nabla \mathrm{A_{CE}})^{-1}|\, |   \mathrm{\hat{A}_{CE}}^\pm| <1 $ for the foliation
 in Theorem \ref{priCE} and the assumption $C |(\nabla   \mathrm{\hat{A}_{CE}}^\pm  + \hat{T})^{-1}| \left( |\nabla k|\,( |\mathrm{Ric}|+|\nabla k|  +|\nabla \nabla k|) +|\nabla \nabla \nabla k| \right) <1  $ in Theorem~\ref{secCMC} are sufficient but not necessary to have the foliation. The necessary condition is that $\alpha= 1 +  \frac{\partial \tau^k}{\partial r}_{|r=0} \langle   e_k, \nu \rangle >0$. If these conditions are not fulfilled, then the proofs of the two theorems still give a regularly centered concentration of constant expansion surfaces surfaces around $p$. 
 \item We have constructed foliations of constant expansion surfaces in two different ways. Note that if the assumptions for one of the Theorems~\ref{priCE} and~\ref{secCMC} are satisfied, the other one is not. The first one requires $\nabla \mathrm{A_{CE}}$ to be invertible and the second one requires $\nabla \mathrm{A_{CE}}=0$.
 \end{enum}
\end{remark}

\section{Uniqueness and nonexistence}
This section closely follows \cite[Section 2]{Ye} with minor modifications. First, we adapt \cite[Lemma 2.1]{Ye} to our situation:

\begin{lemma}\label{lemma 2.1}
Let $\mathcal{F}$ be  a concentration of  regularly centered surfaces at $p\in M$, whose leaves are  constant expansion surfaces  or $STCMC$ surfaces, then it holds:  
\begin{enum}
    \item \label{lemma2.1b} There is a neighborhood $\Omega$ of $p$ together with a  constant $C>1$ such that the mean curvature of $S$ satisfies the bound $C^{-1}(\diam S)^{-1} < H(S)< C(\diam S)^{-1} $ provided that $S$ is a leaf of $\mathcal{F}$ and $S \subset \Omega$. 
    \item \label{lemma2.1a} $\diam S \to  0$ as $\text{dist}(p , S) \to  0$ for leaves $S$ of $\mathcal{F}$.
    \item \label{lemma2.1d} The leaves of $\mathcal{F}$ can be parameterized as a smooth family $S_t$, $0<t\leq 1$ with $S_t \neq S_{t'}$ if $t \neq t'$ and $\lim_{t \to  0} \diam S_t \rightarrow 0$.
\end{enum}
\end{lemma}

\begin{proof}\ 
    \begin{enum}
    \item The upper bound follows from Definition~\ref{regcent}. For the lower bound we use the strong maximum principle for quasi-linear elliptic equations of second order in the following way: Let $\Sigma_1$ and $\Sigma_2$ be two closed oriented surfaces that touch in one point such that at the touching point the normals point into the same direction with $\Sigma_2$ enclosing $\Sigma_1$. Then at the touching point they must satisfy $H(\Sigma_1) \geq H(\Sigma_2)$.

    Let $\diam S$ be the extrinsic geodesic diameter, that is the diameter of the smallest geodesic sphere containing $S$, $S_\rho$, (in particular $\diam S =\diam S_\rho= 2 \rho $).  Then $S_\rho$ touches $S$ in at least one point $q$ with the normals pointing in the same direction. As $S_\rho$ is a geodesic sphere $H(S_\rho) = \frac{n}{\rho} + \mathcal{O} (\rho)$. By choosing $\Omega$ sufficiently small, so that $\rho$ is small, we can ensure $H(S_\rho) >\frac{n-1}{\rho}$. By the maximum principle $H(S)\geq H(S_\rho)>\frac{n-1}{\rho}$ at $q$. Let $\Tilde{q}\in S$ be a different point. For the CE case it follows that 
    \begin{equation}
        H(\Tilde{q}) \geq H(\Tilde{q}) -H(q) + \frac{n-1}{\rho}= \pm (P(q)-P(\Tilde{q})) + \frac{n-1}{\rho}.
    \end{equation}
    Since $P$ is bounded in a neighborhood of $p$ by taking $\Omega $ smaller if necessary we have $\pm (P(\Tilde{q})-P(q)) + \frac{n-1}{\rho} \geq \frac{n-1}{2\rho} $ and therefore $$H(q) \geq \frac{n-1}{2\rho} = \frac{n-1}{\diam S}$$  for any $ \Tilde{q} \in S$. The lower bound for the STCMC case is obtained in an equivalent way.
    \item-- \ref{lemma2.1d}) The rest of the proof is just like in \cite{Ye}.
    \end{enum}
\end{proof}

The next statement is similar to~\cite[Lemma 2.2]{Ye}. The proof in this situation needs a slight modification.

\begin{lemma}
    \label{inesta}
    Let $\mathcal{F}$ be a concentration of surfaces regularly centered at $p\in M$, whose leaves are $STCMC$ surfaces or constant expansion surfaces. For $t$ sufficiently small, $S_t$ is unstable under general variations, i.e. the first eigenvalue of the Jacobi operator $-\Delta-|B|^2 -\mathrm{Ric}(\nu,\nu)$ on $S_t$ is negative. In particular $(H^2-P^2)'(t) \neq 0 $ for STCMC surfaces and $(H\pm P)'(t) \neq 0 $ for constant expansion surfaces.
\end{lemma}
\begin{proof}
    Consider an arbitrary sequence  $t_j \rightarrow 0$ and denote $S_{t_j}$ by $S_j$ and $\diam S_j$ by $d(S_j)$. Consider the auxiliary manifold $(\mathbb{B}_{2r_p}, g_{0,d(S_j)},k_{0, d(S_j)})$ with
    \begin{equation*} 
        g_{0,d(S_j)}= d(S_j)^{-2} \alpha_{d(S_j)}^*(F_0^*(g))
        \qtext{and}
        k_{0, d(S_j)}
        = d(S_j)^{-1} \alpha_{d(S_j)}^*(F_\tau^*(k)) ).
    \end{equation*}
    Let $\Tilde{S_j}=\alpha_{d(S_j)^{-1}} ( F_0^{-1} ( (S_j)))$. Then the diameter of $\Tilde{S_j}$ with respect to the rescaled metric is $1$ and $k_{0, d(S_j)} \rightarrow 0 $ as $j \to \infty $. Since our foliation is regularly centered, the second fundamental forms of $\Tilde{S_j}$ are bounded uniformly in $j\in\IN$. Therefore we can express $\Tilde{S_j}$ locally as the graph of a  function $u_j$ over its tangent space. Note that when studying two points close to each other this functions must agree where they overlap. By Lemma~\ref{lemma 2.1} we have a bound on $|H(\Tilde{S_j})|$ and then by the ellipticity of the mean curvature equation  we obtain Schauder estimates on $u_j$, 
    \begin{equation*}
        \|u_j \|_{\mathcal{C}^{2,\frac{1}{2}}} 
        \leq C  \big(\|u_j \|_{\mathcal{C}^{0}} 
            + \|H(\text{graph} \, u_{j}) \|_{\mathcal{C}^{0,\frac{1}{2}}} \big).
    \end{equation*}
    The bound on the diameter gives a uniform bound on $\|u_j \|_{\mathcal{C}^{0}}$. By the CE or STCMC equation, $\|\nabla H(\text{graph} \, u_{j}) \|_{\mathcal{C}^{0}}$ is bounded by $\|\nabla P(\text{graph} \, u_{j}) \|_{\mathcal{C}^{0}} $ and this is uniformly bounded in $j$ since  $|k|$ and $|\nabla k|$ are bounded in a neighbourhood of $p$ and since the second fundamental forms of $\text{graph} \, u_{j}$ are uniformly bounded. 
    Thus, for a positive constant $C$, it follows that  $\|H(\text{graph} \, u_{j}) \|_{\mathcal{C}^{0,\frac{1}{2}}}< C $ for all $j\in \IN$ and by Arzela Ascoli $(u_j)$ has a convergent sub-sequence so that local pieces of $\Tilde{S_j}$ sub-converge smoothly. 
    
    From the scaling of $k_{0, d(S_j)}$ it follows that $P(\Tilde{S_j}) \to 0$ as $j \to \infty$. Then as $(H^2-P^2)(\Tilde{S_j})=C_j $ or  $(H\pm P)(\Tilde{S_j})=C_j $ where $C_j$ is a constant, the bounds of Lemma~\ref{lemma 2.1}\, (\ref{lemma2.1b}) imply that the sequence $(C_j)_{j\in\IN}$ is bounded and bounded away from $0$. Thus by selecting a further sub-sequence we can assume that $\lim_{j\to\infty} C_j=C$ exists and is non-zero. Thus $\lim_{j\to \infty} (H^2-P^2)(\Tilde{S_j}) = \lim_{j\to \infty} H(\Tilde{S_j})^2=C$ for the STCMC case or $\lim_{j\to \infty} (H\pm P)(\Tilde{S_j}) = \lim_{j\to \infty} H(\Tilde{S_j})=C$ for the constant expansion case.

    If we consider such a convergent sub-sequence  $u_{j'} $ and some $p_{j'} \in  \Tilde{S_{j'}}$ such that each $\text{graph} \, u_{j'} $ is a neighbourhood of $p_{j'}$ we  have that $\lim_{j' \to \infty} H( \text{graph}(u_{j'}))=\lim_{j'\to \infty} C_{j'} = C$, so that the graphs of the functions $u_{j'}$ converge to pieces of CMC surfaces in Euclidean space. 
    
    Assume that one such local sheet $\Sigma_1$ passes through  $p\in\R^{n+1}$ and by eventually selecting a further sub-sequence on can construct a different sheet $\Sigma_2$ also passing through $p$. Then since the $\tilde S_j$ are embedded, $\Sigma_1$ and $\Sigma_2$ must
    be ordered near $p$ in the sense that $\Sigma_1$ is locally on one side of $\Sigma_2$. Such a situation is possible. However, if one can construct three sheets $\Sigma_1$, $\Sigma_2$ and $\Sigma_3$ meeting at $p$ in a locally ordered fashion, then it is also possible to construct three sheets such that $\Sigma_2$ is between $\Sigma_1$ and $\Sigma_3$, such that the outward normal of $\Sigma_1$ and $\Sigma_3$ agree at $p$ an such that the outward normal of $\Sigma_2$ points into the opposite direction at $p$ than the outward normals of $\Sigma_1$ and $\Sigma_2$ (cf. figure~\ref{fig:graphs}). This situation is incompatible with the strong maximum principle and the fact that the $\Sigma_i$ have the same nonzero constant mean curvature (Lemma \ref{lemma 2.1}). First, the strong maximum principle implies that $\Sigma_1$ and $\Sigma_2$ agree in a neighborhood of $p$ and thus also the middle sheet $\Sigma_2$ must agree with $\Sigma_1$. Since both surfaces have different orientations, this would mean that they have constant mean curvature zero, which is a contradiction. 
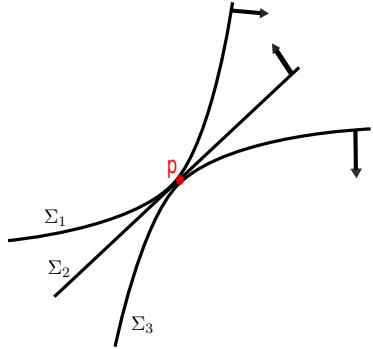
\begin{figure}[h]
    \label{fig:graphs}
    \centering
    \resizebox{.3\linewidth}{!}{
    \begin{picture}(0,0)%
    \includegraphics[width=200pt]{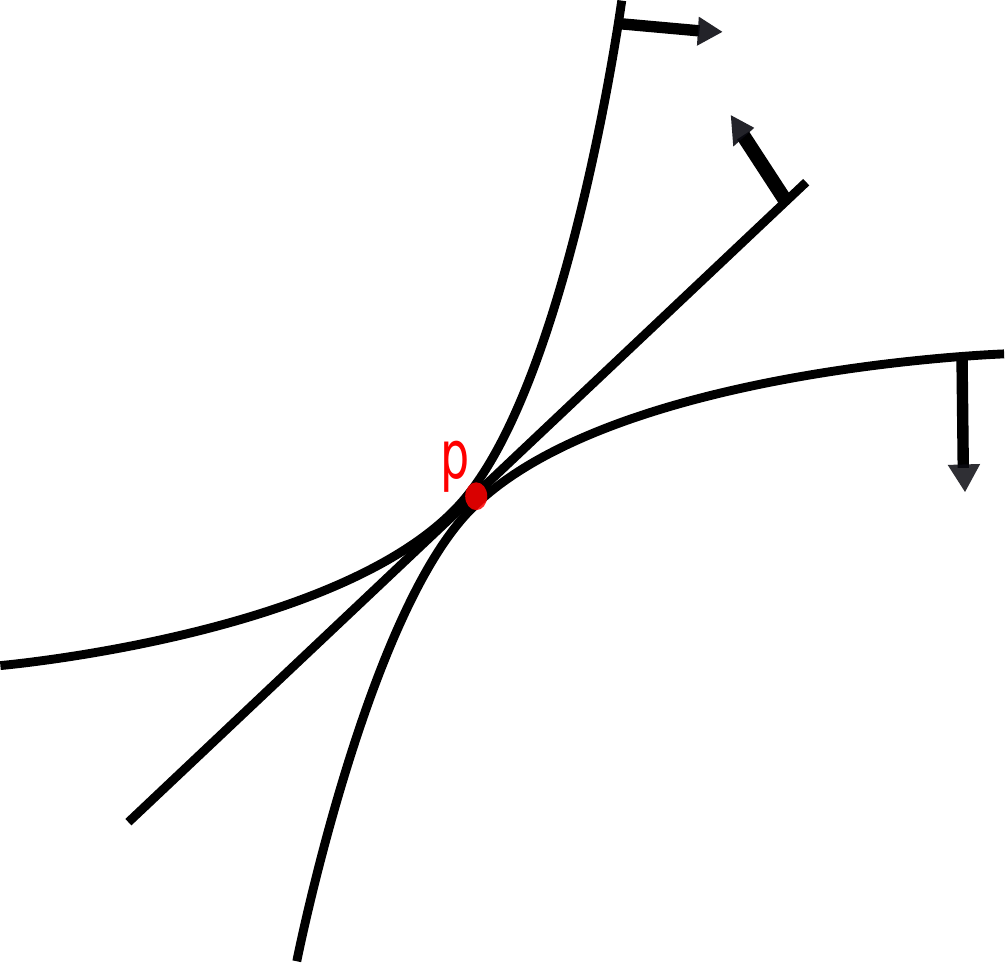}%
    \end{picture}%
    \setlength{\unitlength}{1pt}%
    \begin{picture}(200,200)(0,0)
    \put(20,67){\makebox(0,0)[lb]{\color[rgb]{0,0,0}$\Sigma_1$}}%
    \put(22,39){\makebox(0,0)[lb]{\color[rgb]{0,0,0}$\Sigma_2$}}%
    \put(68,8){\makebox(0,0)[lb]{\color[rgb]{0,0,0}$\Sigma_3$}}%
    \end{picture}%
}
\caption{Three pieces $\Sigma_1$, $\Sigma_2$, and $\Sigma_3$ meeting at $p$ with alternating orientations.}
\end{figure}
   
    Consequently, there can be no more than two sheets passing through a single point $p\in\R^{n+1}$. This means that the $\tilde S_j$ subconverge to an Alexandrov-embedded hypersurface in $\R^{n+1}$ with positive constant mean curvature. By Alexandrov's Theorem the limit must be a round sphere. Round spheres are unstable in the Euclidean metric (when allowing variations that don't preserve volume) and therefore for $j$ large enough $\Tilde{S_j}$ is unstable.

    The instability of the surfaces directly implies that any normal variation of $\Tilde{S_j}$ satisfies $H'(\Tilde{S_j})\neq 0 $, now as $P(\Tilde{S_j})= d(S_j) P_g(\Tilde{S_j}) \to 0$ for $j\to \infty$ we also have that  the normal variation $P'(\Tilde{S_j})$ is small for large $j$, and therefore    $ (H^2-P^2)'(\Tilde{S_j}) \neq 0 $ for STCMC case or $ (H\pm P)'(\Tilde{S_j}) \neq 0  $ for the constant expansion case when considering  $j$ large. Then as this applies for any sequence $t_j$ and $g_{0,d(S_j)} $ is a conformal metric we have that  $(H^2-P^2)'(t) \neq 0 $ or $(H\pm P)'(t) \neq 0 $ for $t$ sufficiently small.
\end{proof}
\begin{remark}
 The stability used in this last lemma is not the standard stability under volume preserving variations, but a general one allowing changes in the volume, under such unconstrained variations round spheres are unstable CMC surfaces in Euclidean space.   
\end{remark}

By the previous result, we have that $(H\pm P)(t)  $ and $(H^2 - P^2)(t)$ increase as $t$ decreases. Furthermore, since $P$ is bounded in a neighborhood of $p$,  if there is some $t_0$ such that $(H\pm P)(t_0)<0  $ or $(H^2 - P^2)(t_0)<0  $  then there exists $t_+<t_0$ such that $(H\pm P)(t)>0$ or $(H^2 - P^2)(t_0)>0$  $\forall t<t_+$. Therefore, we can always assume that these surfaces have  positive expansion, or positive space-time mean curvature and we can introduce the parameter $$\frac{n}{(H\pm P)(t)}=r= \frac{n}{\sqrt{(H^2-P^2)(t)}} .$$

A direct application of \cite[Corollary 2.1 and Lemma 2.3]{Ye}  with practically no change  gives a parametrization of the  $S_r$ (the leaves of a possible restriction of $\mathcal{F}$) given by graphs over geodesic spheres with perturbed centers, that is  $F_{\tau(r)} (\alpha_r (S^n_{\varphi^*(r)}))$ where    $\varphi$ is a function on the unit sphere $\mathbb{S}^n$ with the properties that $\pi(\varphi^*(r))=0$ and $\lim_{r \to 0} \|\varphi^*(r)\|_{\mathcal{C}^3}=0 $, and $\tau(r) \in \mathbb{R}^{n+1}$ satisfies  $\tau(0)=0$.
\begin{theorem}[Nonexistence and uniqueness, STCMC]\label{unicSTCMC}\ 
\begin{enum} 
\item \label{unicSTCMC-a} If at a point $p$, $\mathrm{A_{ST}} \neq 0$, then there exists no constant STCMC foliation regularly centered at $p$.

\item \label{unicSTCMC-b} Assume that at $p$ $\mathrm{A_{ST}}=0$, $\nabla \mathrm{A_{ST}}$ is invertible and that the foliation $ \mathcal{F}$ of Theorem \ref{priSTCMC}  exists. If  $ \mathcal{F}_2$ is a STCMC foliation regularly centered at $p$, then either $\mathcal{F}$ is a restriction of $\mathcal{F}_2$ or $\mathcal{F}_2$ is a restriction of $\mathcal{F}$.

\item \label{unicSTCMC-c} The claims (\ref{unicSTCMC-a}) and (\ref{unicSTCMC-b}) also hold if instead of  foliations we consider concentrations of regularly centered surfaces around $p$.
 \end{enum}
\end{theorem}
\begin{proof}
 Suppose $\mathcal{F}$ is a STCMC foliation regularly centered around $p$ and $F_{\tau(r)} (\alpha_r (S^n_{\varphi^*(r)}))=S_r$ the parametrization found before.
 \begin{enum} 
 \item Since $(H^2-P^2)(r)=\frac{n^2}{r^2}$ we have that in the rescaled manifold $(\mathbb{B}_{2r_p}, g_{\tau,r},$ $k_{\tau, r} )$ from above,    the rescaled surfaces satisfy $(H^2-P^2)(r,\tau(r), \varphi^*(r))=n^2$, and therefore  
    \begin{equation*} 
        \Tilde{\pi}\left((H^2 -P^2)(r, \tau(r),  \varphi^*(r) )\right)=0, 
        \quad\text{and}\quad
        \Tilde{\pi}^\bot\left((H^2 -P^2)(r, \tau(r),  \varphi^*(r) )-n\right)=0
    \end{equation*}
    Since $\varphi^* = \mathcal{O}(r)$ the expression $\frac{\varphi^*}{r}$ uniformly bounded for small $r$. As in \eqref{1a} we have 
    \begin{equation}\label{1b}
        \begin{split}
       2 n L\varphi^*     
        &=\frac{2n}{3}  \mathrm{Ric}^{\tau}_{ij}(0)x^i x^j r^2  + \frac{2n}{4} \mathrm{Ric}^{\tau}_{ij;k}(0)x^ix^jx^k r^3+ r^2P_g(r,\tau(r),0)^2    \\
        &\quad - r^2 \int_0^1 \int_0^1 t (H^2 -P^2)_{\varphi \varphi}(sr, \tau,st  \varphi^* )  \frac{\varphi^*}{r} \frac{\varphi^*}{r} \ds \dt\\
        &\quad - r^3 \int_0^1 \int_0^1 \int_0^1 s (H^2 -P^2)_{\varphi r r}(usr, \tau,ust  \varphi^* )   \frac{\varphi^*}{r} \du \ds \dt\\
        &\quad - r^3 \int_0^1 \int_0^1 \int_0^1 s t (H^2 -P^2)_{\varphi \varphi r} (usr, \tau,ust  \varphi^* )  \frac{\varphi^*}{r} \frac{\varphi^*}{r} \du \ds \dt   + \mathcal{O}(r^2) \\
        &=:r^2 f(r)
        \end{split}
    \end{equation}
    where $f(r)$ is uniformly bounded (in all scale invariant norms) for small $r$ . Thus $\varphi^*$ solves the elliptic PDE $2n L\varphi =r^2 f(r)$ in $K^\bot$. From Schauder estimates it follows that $\|\varphi^* \|_{\mathcal{C}^{2,\frac{1}{2}}} \leq C r^2$. Considering the projection to the kernel and dividing by $r ^3$ like in \eqref{kern} we obtain 
    \begin{equation}
        \begin{split}
            0=& -\frac{n|\mathbb{S}^n|}{(n+1)(n+3) } \mathrm{Sc}_{,l} e_l - \frac{1}{r^3} \Tilde{\pi}\left(P^2 (r,\tau(r), 0)\right) + \mathcal{O}(r^2) \\
            & +\Tilde{\pi} \Bigg( \int_0^1 \int_0^1 t (H^2 -P^2)_{\varphi \varphi}(sr, \tau,st  \varphi^* )  \frac{\varphi^*}{r^2} \frac{\varphi^*}{r} \ds \dt\\
            &\qquad+  \int_0^1 \int_0^1 \int_0^1 s (H^2 -P^2)_{\varphi r r}(usr, \tau,ust  \varphi^* )   \frac{\varphi^*}{r} \du \ds \dt\\
            &\qquad+  \int_0^1 \int_0^1 \int_0^1 s t (H^2 -P^2)_{\varphi \varphi r} (usr, \tau,ust  \varphi^* )  \frac{\varphi^*}{r} \frac{\varphi^*}{r} \du \ds \dt \Bigg)   \\
        \end{split}
    \end{equation}
    Then as $\|\frac{\varphi^*}{r} \|_{\mathcal{C}^{0}} \to 0$ as $r \to 0$, we have that in the limit $r\to 0$ it follows that
    \begin{equation}
        \begin{split}
            0=& -\frac{n|\mathbb{S}^n|}{(n+1)(n+3) } \mathrm{Sc}_{,l} e_l -  \Tilde{\pi}\left.\left(\frac{1}{r^3} P^2 (r,\tau(r), 0)\right)\right|_{r=0} \\
        \end{split}
    \end{equation}
    and  this  implies that $\mathrm{A_{ST}}=0$ at $p$.
 
 \item Setting $\varphi(r):= \varphi^*(r) r^{-2}$ we have that 
    \begin{equation*} 
        \Tilde{\pi}\left((H^2 -P^2)(r, \tau(r),  r^2 \varphi(r) )\right)=0
        \qtext{and}
        \Tilde{\pi}^\bot\left((H^2 -P^2)(r, \tau(r),   r^2\varphi(r) )-n\right)=0.
    \end{equation*}
    Since $\tau(0)=0$ and  by~\eqref{1b} it follows that  
     $\varphi(0)$ is the solution of  $L\varphi_0 = \frac{1}{3} Ric^0_{ij}(0) x^ix^j+ \frac{1}{2n}P^2_g(0,0,0)$. Thus the uniqueness part
     of the implicit function theorem used in the proof of Theorem~\ref{priSTCMC} implies that the functions $\varphi(r)$ and $\tau(r)$
     agree with the ones found in Theorem~\ref{priSTCMC}  in a neighborhood of $r=0$.
 \item Finally note that here and also in Lemmas~\ref{lemma 2.1} and~\ref{inesta} we have not used the foliation property for our surfaces,
      only that they are a regularly centered concentration of surfaces around $p$. 
\end{enum}
\end{proof}

\begin{theorem}[Nonexistence and uniqueness, CE]\label{unicCE}\ 
\begin{enum}
\item\label{unicCE-a} If at a point $p$, $\mathrm{A_{CE}} \neq 0$ or $k \neq 0$, then there exist no constant expansion 
    foliation regularly centered at $p$.
\item\label{unicCE-b} Assume that at $p$, $\mathrm{A_{CE}}=0$, $k=0 $, $\nabla \mathrm{A_{CE}}$ is invertible and that the foliation 
    of Theorem~\ref{priCE} $ \mathcal{F}$ exist. If  $ \mathcal{F}_2$ is a constant expansion foliation regularly centered at $p$,
    then either $\mathcal{F}$ is a restriction of $\mathcal{F}_2$ or $\mathcal{F}_2$ is a restriction of $\mathcal{F}$.
\item\label{unicCE-c}  Assume that at $p$, $\mathrm{A_{CE}}=  \mathrm{\hat{A}_{CE}}^\pm =0$, $k= \nabla \mathrm{A_{CE}}=0$, 
    $\Hess\mathrm{A_{CE}}=0 $ and $\nabla   \mathrm{\hat{A}_{CE}}^\pm  + \hat{T}$  is invertible, and that the foliation $ \mathcal{F}$ of
    Theorem~\ref{secCMC}  exists. If  $ \mathcal{F}_2$ is a constant expansion foliation regularly centered at $p$, 
    then either $\mathcal{F}$ is a restriction of $\mathcal{F}_2$ or $\mathcal{F}_2$ is a restriction of $\mathcal{F}$.
\item\label{unicCE-d} The claims (\ref{unicCE-a}), (\ref{unicCE-b}), and (\ref{unicCE-c}) also hold if instead of 
    foliations we consider concentrations of regularly centered surfaces around $p$. 
\end{enum}
\end{theorem}

\begin{proof}
    The proof is completely analogous to Theorem~\ref{unicSTCMC}. Suppose $\mathcal{F}$ is a STCMC foliation regularly centered around $p$ and $F_{\tau(r)} (\alpha_r (S^n_{\varphi^*(r)}))=S_r$ the parameterization of before. 
    \begin{enum} 
    \item As in the previous proof we have 
        \begin{equation*} 
            \Tilde{\pi}\left((H \pm P)(r, \tau(r),  \varphi^*(r) )\right)=0
            \qtext{and} 
            \Tilde{\pi}^\bot\left((H \pm P)(r, \tau(r),  \varphi^*(r) )-n\right)=0.
        \end{equation*}
        Hence
        \begin{equation}\label{1c}
        \begin{split}
             L\varphi^*     &=\frac{1}{3}  \mathrm{Ric}^{\tau}_{ij}(0)x^i x^j r^2 
            +  \frac{1}{4}  \mathrm{Ric}^{\tau}_{ij;l}(0)x^i x^j x^l r^3  \pm r P_g(r,\tau(r),0)  
            -  (H \pm P)_{r\varphi} (0,\tau,0 )\frac{\varphi^*}{r} r^2\\
            &- r^2 \int_0^1 \int_0^1 t (H \pm P)_{\varphi \varphi}(sr, \tau,st  \varphi^* )  \frac{\varphi^*}{r} \frac{\varphi^*}{r} \ds \dt\\
            &- r^3 \int_0^1 \int_0^1 \int_0^1 s (H \pm P)_{\varphi r r}(usr, \tau,ust  \varphi^* )   \frac{\varphi^*}{r} \du \ds \dt\\
            &- r^3 \int_0^1 \int_0^1 \int_0^1 s t (H \pm P)_{\varphi \varphi r} (usr, \tau,ust  \varphi^* )  \frac{\varphi^*}{r} \frac{\varphi^*}{r} \du \ds \dt   + \mathcal{O}(r^2) \\
            &=:r^2 f(r)
        \end{split}
    \end{equation}
    where $f(r)$ is uniformly bounded (in all scale invariant norms) for small $r$. Thus $\varphi^*$ solves the elliptic PDE 
    $ n L\varphi=r^2 f(r)  $ in $K^\bot$. Therefore, by Schauder estimates $\|\varphi^* \|_{\mathcal{C}^{2,\frac{1}{2}}} \leq C r^2$. Considering the projection to the kernel and dividing by $r ^2$ we obtain 
    \begin{equation}
        \begin{split}
            0=&-\frac{|\mathbb{S}^n|}{2(n+1)(n+3) } r^2 \,\mathrm{Sc}_{,l} e_l \pm  \Tilde{\pi}\left(\frac{1}{r^2} P (r,\tau(r), 0)\right)  
            - \Tilde{\pi}\left( (H \pm P)_{r\varphi} (0,\tau,0 )\frac{\varphi^*}{r}  \right)\\
            & +\Tilde{\pi} \big( \int_0^1 \int_0^1 t (H \pm P)_{\varphi \varphi}(sr, \tau,st  \varphi^* )  
                \frac{\varphi^*}{r^2}     \frac{\varphi^*}{r} \ds \dt\\
            &+  \int_0^1 \int_0^1 \int_0^1 s (H \pm P)_{\varphi r r}(usr, \tau,ust  \varphi^* )   \frac{\varphi^*}{r} \du \ds \dt\\
            &+  \int_0^1 \int_0^1 \int_0^1 s t (H \pm P)_{\varphi \varphi r} (usr, \tau,ust  \varphi^* )  \frac{\varphi^*}{r} \frac{\varphi^*}{r} \du \ds \dt \big)  + \mathcal{O}(r^2) \\
        \end{split}
    \end{equation}
    Then as $\|\frac{\varphi^*}{r} \|_{\mathcal{C}^0} \to 0$ as $r \to 0$, we have that in the limit $r\to 0$ it follows that 
    \begin{equation}
        \begin{split}
            0=&  \Tilde{\pi}\left( \frac{1}{r^2}P (r,\tau(r), 0)\right)_{|r=0} \\
        \end{split}
    \end{equation}
    and this implies that $\mathrm{A_{CE}}=0$ at p.
 
    Consider the the projection to $K^\bot$:
    \begin{equation*}
        \begin{split}
        0=\pi^\bot \left(\frac{(H \pm P)(r, \tau, r^2 \varphi )-n}{r^2}\right)&=   L\varphi^*- \pi^\bot\left(  \tfrac{1}{3} \mathrm{Ric}^0_{ij}(0) x^ix^j  \mp  \tfrac{1}{r} P_g(r,\tau(r),0)\right)  
        \end{split}
    \end{equation*}
    Since $\|\frac{\varphi^*}{r}\|_{C^0}\to 0$ as $r\to 0$ by Elliptic Regularity also $\lim_{r \to 0} \|L \frac{\varphi^*}{r}\| < \infty $. This implies that 
    \begin{equation*} 
        \lim_{r \to 0} \left| \frac{1}{3} \mathrm{Ric}^0_{ij}(0) x^ix^j  \mp  \frac{1}{r} P_g(r,\tau(r),0) \right|< \infty
    \end{equation*}
    so that $ \lim_{r \to 0} \frac{P_g(r, \tau(r),0)}{r} <\infty$. This implies that at $p$, $\tr k - k_{ij} x^ix^j=0$ for all $x\in\R^n$, which in turn implies $k=0$. 
 
    \item and \ref{unicCE-c}) Setting $\varphi(r):= \varphi^*(r) r^{-2}$ it follows that 
    $\Tilde{\pi}\left((H \pm P)(r, \tau(r),  r^2 \varphi(r) )\right)=0$ and $\quad $ $\Tilde{\pi}^\bot\left((H \pm P)(r, \tau(r),   r^2\varphi(r) )-n\right)=0$. Furthermore, since $\tau(0)=0$ and since by \eqref{1c}
    $\varphi(0)$ is given by \eqref{varphi}, the uniqueness part of the implicit function theorem used in Theorem~\ref{priCE} 
    or in Theorem~\ref{secCMC} implies that the functions $\varphi(r)$ and $\tau(r)$ must agree in a neighborhood of $r=0$
    with the ones found in the Theorems.
    \stepcounter{enumi}
    \item In Lemmas~\ref{lemma 2.1}, \ref{inesta} and here we have not used the foliation property for our surfaces, only that they are          regularly centered. 
    \end{enum}
\end{proof}

\section{Physical relevance} In dimension $3$ we were expecting to find a physical meaningful quantity related with the concentration of STCMC surfaces, a result somehow similar the result found when studying the small sphere limit of quasi local energies  on surfaces approaching a point $p$ in a space-time along the null cone: 
\begin{equation}
    \lim_{r \to 0} \frac{\mathcal{E}(S_r)}{r^3} \sim T(N,N)  = \mathrm{Sc}_p + (\tr k)^2 - |k|^2. 
\end{equation}
Here $T$ is the stress energy tensor and $N$ is the time-like vector determining the time cone.  In our case  we found  that at the point $p$ where there is a  concentration of STCMC surfaces the local STCMC $1$-form must vanish. This is the following $1$-form vanishing
\begin{equation}
    \begin{split}
    A_\text{ST}(V)  = \nabla_V \mathrm{Sc} + \frac{1}{7}&\Big[ \nabla_V (|k|^2 )   +22\nabla_V \Big(\frac{(\tr k)^2}{2}\Big)   +4 \diver \left( \langle k,   k(V,\cdot )\rangle \right)
    - 12  \diver \left(\tr k\cdot k(V,\cdot )\right)\Big]
    \end{split}
\end{equation}
for $V\in T_pM$. Actually the local existence of concentration of such surfaces is determined by $A_\text{ST}$. We  have not been able to relate $A_\text{ST}$ to any known physical quantity. However we would like to make some comments about this result:
\begin{enum}
    \item One could argue that our method for constructing these surfaces is lacking of information from the space-time, since our surfaces are constructed as deformations of geodesic spheres of a hyper-surface in space-time. Therefore there is no direct relation to geodesic spheres in space-time. However thanks to Theorem \ref{unicSTCMC}  we have uniqueness of these surfaces and an existence condition, therefore any concentration of regularly centered STCMC surfaces will lead to the same result.
    \item It is important to note that in the asymptotically Euclidean case \cite{STCMC} the STCMC surfaces encode the information about the center of mass, however the physically relevant quantity is the center of these surfaces, in our case the center of the surfaces would be trivially $p$.
    \item Note that there is a direct way to  construct manifolds on which the     foliations of Theorem \ref{priSTCMC} exist. For this, consider a manifold with a point $p'$ which is a non-degenerate critical point of the scalar curvature and let $k$ be an arbitrary 2-tensor. We proceed, using the implicit function theorem in the following way. Let
    \begin{equation}
        \begin{split}
        G(p, \eps) &:=A_\text{ST}(p,\eps)\\
        &= \nabla \mathrm{Sc} + \frac{1}{7} \Big[ \nabla (|\eps k|^2 )   +22\nabla \Big(\frac{(\eps \tr  k)^2}{2}\Big)   +4 \diver \left( \langle \eps k,\eps     k \rangle \right)
        - 12  \diver \left(\eps^2 \tr   k\cdot    k\right)\Big]\\
        \end{split}
    \end{equation}
    By assumption  $\frac{\partial}{\partial p} G(p', 0)= \Hess\mathrm{Sc}_{p'}$ is invertible. Hence, there exists $\eps_1>0$ and a function $p:(-\eps_1,\eps_1)\to M$ with $p(0)=p'$ such that
    $G(p(\eps),\eps)=0$ for all $\eps\in(-\eps_1,\eps_1)$. Since $\frac{\partial}{\partial p} G(p', 0)$ is invertible, there is $\eps_0\in(0,\eps_1)$ such that $\frac{\partial}{\partial p} G(p(\eps),\eps)$ is invertible for all $\eps\in(-\eps_0,\eps_0)$. 

    By choosing $\eps_0$ smaller if necessary, by continuity one can also ensure that the bound
    \begin{equation*}
        C |(\nabla \mathrm{A_{ST}}(\eps))^{-1}|\, (\eps^2|k|^2 + |\mathrm{Ric}|  ) \eps^2 |k| |\nabla k|<1
    \end{equation*}
    in the assumption of Theorem \ref{priSTCMC} is satisfied. Therefore we have a STCMC foliation at $p(\epsilon)$ in the manifold $(M,g,\eps k)$ for any  $\eps\in(-\eps_0,\eps_0)$  (note that in this manifold  $\mathrm{A_{ST}}(\eps)=\mathrm{A_{ST}}$ ).
    \end{enum}
\appendix

\section{ technical Lemma}

\begin{lemma}
The components of the position vector of a point in the sphere $\mathbb{S}^n$ satisfy
\begin{equation*}
        \begin{split}
            \int_{\mathbb{S}^n} x^i x^j d\mu =  \frac{|\mathbb{S}^n|}{n+1} \delta_{ij},
        \end{split}
    \end{equation*}
    
    \begin{equation*}
        \begin{split}
            \int_{\mathbb{S}^n} x^i x^j x^k x^l d\mu =  \frac{|\mathbb{S}^n|}{(n+1)(n+3)} (\delta_{ij} \delta_{kl} + \delta_{ik} \delta_{jl} + \delta_{il} \delta_{jk}),
        \end{split}
    \end{equation*}

and 
\begin{equation*}
        \begin{split}
            \int_{\mathbb{S}^n} x^i x^j x^k x^lx^p x^q d\mu =  \frac{|\mathbb{S}^n|}{(n+1)(n+3)(n+5)}&  ( \delta_{ij} \delta_{kl} \delta_{pq} + \delta_{ij} \delta_{kp} \delta_{lq}+ \delta_{ij} \delta_{kq} \delta_{lp}\\
&+ \delta_{ik} \delta_{jl} \delta_{pq}+ \delta_{ik} \delta_{jp} \delta_{lq} + \delta_{ik} \delta_{jq} \delta_{lp}\\
&+ \delta_{il} \delta_{jk} \delta_{pq} + \delta_{il} \delta_{jp} \delta_{kq} + \delta_{il} \delta_{jq} \delta_{kp}\\
&+ \delta_{ip} \delta_{jk} \delta_{lq} + \delta_{ip }\delta_{jl} \delta_{kq} + \delta_{ip} \delta_{jq} \delta_{kl}\\
&+ \delta_{iq} \delta_{jk} \delta_{lp} + \delta_{iq} \delta_{jl} \delta_{kp} + \delta_{iq} \delta_{jp} \delta_{kl} )\\
        \end{split}
    \end{equation*}

\end{lemma}

\bibliographystyle{amsplain}
\bibliography{Lit_new}
\end{document}